\documentclass[a4paper,12pt]{amsart}
\usepackage{amsmath,amsfonts,amssymb,amsthm}
\usepackage{graphics}
\usepackage{epsfig}
\usepackage{esint}
\usepackage{shuffle}
\usepackage{istgame}
\usepackage[hidelinks]{hyperref}

\usepackage{mathrsfs}
\usepackage{enumitem}

\numberwithin{equation}{section}

\usepackage{vmargin}
\setpapersize{A4}
\setmargins{2.5cm}{1.5cm}{16.5cm}{23.42cm}{10pt}{1cm}{0pt}{2cm}

\newcommand{\ignore}[1]{}

\newcommand{\dum}[1]{\mathsf{#1}} 

\newtheorem{definition}{Definition}[section]
\newtheorem{proposition}[definition]{Proposition}
\newtheorem{theorem}[definition]{Theorem}
\newtheorem{remark}[definition]{Remark}
\newtheorem{lemma}[definition]{Lemma}
\newtheorem{corollary}[definition]{Corollary}

\newcommand{\N}{\mathbb{N}}

\newcommand{\Q}{{\mathbb Q}}
\newcommand{\R}{\mathbb{R}}

\newcommand{\n}{\mathbf{n}}
\newcommand{\m}{\mathbf{m}}
\newcommand{\z}{\mathsf{z}}
\newcommand{\0}{\mathbf{0}}
\newcommand{\T}{\mathsf{T}}

\newcommand{\pol}[1]{|#1|_p}

\newcommand{\Ppol}{\mathsf{P}}

\makeatletter
\newcommand{\opnorm}{\@ifstar\@opnorms\@opnorm}
\newcommand{\@opnorms}[1]{%
	\left|\mkern-1.5mu\left|\mkern-1.5mu\left|
	#1
	\right|\mkern-1.5mu\right|\mkern-1.5mu\right|
}
\newcommand{\@opnorm}[2][]{%
	\mathopen{#1|\mkern-1.5mu#1|\mkern-1.5mu#1|}
	#2\mathclose{#1|\mkern-1.5mu#1|\mkern-1.5mu#1|}
}
\makeatother

\usepackage{color}
\definecolor{darkred}{rgb}{0.9,0.1,0.1}
\definecolor{darkblue}{rgb}{0,0,0.7}
\definecolor{darkgreen}{rgb}{0,0.5,0}

\newcounter{step}
\newcounter{refstep} 

\title[A tree-free approach to regularity structures]{A tree-free approach to regularity structures: the regular case for quasi-linear equations} 
\author{Pablo Linares}
\author{Felix Otto}

\begin{document}
\begin{abstract}
	We give a motivation and gentle introduction into the
	regularity structure and model introduced in \cite{OSSW21},
	which fall into the framework of Hairer,
	but have a greedier index set than the one given by trees.
	We do this here for a simple quasi-linear parabolic equation and 
	assume that the driving noise is so regular that no renormalization
	is needed. 
	
	\smallskip
	
	We introduce the abstract model space $\mathsf{T}$ and its grading,
	the pre-model $\mathbf{\Pi}$, the centered model $\Pi_x$, the structure group $\mathsf{G}$, 
	and the re-centering transformations $\Gamma_{xy}$. Using integration and reconstruction,
	we establish the desired estimates on $\Pi_x$ and $\Gamma_{xy}$, 
	which here are deterministic since we deal with the regular case.
\end{abstract}
	\maketitle

\tableofcontents

\section{Introduction}
In these notes, we build a regularity structure and a model for the quasi-linear PDE
\begin{align}\label{eq}
	(\partial_2-\partial_1^2)u=a(u)\partial_1^2u + \xi,
\end{align}
which we think of as uniformly parabolic, where $\xi$ is space-time periodic with vanishing average in the unit cell.
The purpose of regularity structures introduced in \cite{Ha14} is to provide the
framework for a solution theory for \eqref{eq} with a random -- and thus typically rough 
-- driver $\xi$, like space-time white noise. 
The regularity of noise and solutions is captured on the scale of (negative) 
H\"older spaces\footnote{since the interest is in stationary noise,
there is no loss in using these uniform spaces},
with respect to the parabolic Carnot-Carathéodory metric 
\begin{equation}\label{dist01}
	|y-x|:=|y_1-x_1|+\sqrt{|y_2-x_2|}.
\end{equation}
If $\xi\in C^{\alpha-2}$ we cannot expect better than $u\in C^\alpha$, so that
the product $a(u)\partial_1^2u$ can be given a classical sense only when the function
$a(u)$ is more regular than the distribution $\partial_1^2u$ is irregular,
which translates into $\alpha+(\alpha-2)>0$, i.~e.~$\alpha>1$. Note that space-time
white noise corresponds to $\alpha=\frac{1}{2}-$. It is for the treatment of the
singular case, which corresponds to $\alpha<1$ here, that regularity structures were created. 
For an introduction to the theory of regularity structures, we refer the reader to \cite{Ha16}.

\medskip

We follow the tree-free approach to regularity structures introduced in \cite{OSSW21} 
for quasi-linear equations of the type \eqref{eq}. 
This framework was algebraically characterized in \cite{LOT21}; 
the renormalization and stochastic estimates were provided in \cite{LOTT21},
all for the full singular subcritical regime $\alpha\in (0,1)$.
The goal of these notes is to provide a careful motivation and elementary construction of
a regularity structure and a model within this approach.
For the sake of clarity, we will restrict to the range\footnote{ The construction of Section \ref{sec:fullmodel} would work for any $\alpha>0$ after minor modifications; it is for Section \ref{sec:analyticalproperties} that we require $\alpha\in (1,2)$, although the more regular case $\alpha>2$ can be treated more easily.} $\alpha\in (1,2)$ 
which does not require a renormalization. 
These notes are self-contained; no knowledge of \cite{OSSW21,LOT21,LOTT21} is required.

\medskip

As the labeling of the variables in \eqref{eq}, $x_1$ for space and $x_2$ for time, 
alludes to, we treat the parabolic equation like an elliptic one -- in line with a good
tradition in Schauder theory. In fact, when it comes to this text, there is no
essential difference between the parabolic and elliptic operators,
$\partial_2u-a(u)\partial_1^2u$ and $-\partial_2^2u-a(u)\partial_1^2u$, respectively. 
Also, we believe to be able to treat 
the higher-dimensional version with a tensorial non-linearity $a$ in a similar way.  
In view of this parallel treatment of parabolic and elliptic problems, we allow ourselves
to assume periodicity in both variables $x_1$ and $x_2$.

\medskip

We essentially follow the set-up of regularity structures, also in terms of language and notation.
The main difference resides in a smaller abstract model space $\mathsf{T}$, 
which is indexed not by trees but by multi-indices\footnote{in fact, our 
$\beta$'s correspond to certain linear combinations of trees, cf. \cite[Sections 6 and 7]{LOT21}}. 
This index set of multi-indices comes from (formally) endowing the manifold of solutions 
with coordinates and taking partial derivatives with respect to these, giving rise to
the (pre-)model $\mathbf{\Pi}$. The
coordinates come in two sets: The first set $\{\mathsf{z}_k\}_{k\in\mathbb{N}_0}$
parameterizes the nonlinearity $a$ via a power series representation of the latter.
Hence in the spirit of rough paths for stochastic ODEs, we treat all nonlinearities in parallel. 
A second set of coordinates $\{\mathsf{z}_{\bf n}\}_{{\bf n}\in\mathbb{N}_0^2\setminus\{{\bf 0}\}}$
is required for a PDE, and parameterizes a large-scale polynomial behavior. 
Rather than on the model space $\mathsf{T}$, 
our focus is more on its algebraic dual $\mathsf{T}^*$, 
where the model takes values (cf. \cite[Definition 3.3]{Ha16}) in.
As a consequence, we think of the structure group $\mathsf{G}$ as acting on $\mathsf{T}^*$.

\medskip

Next to these constructions, which have a mostly algebraic character, 
we establish estimates in line with the axioms \cite[(3.2)]{Ha16}. 
While a similar inductive principle is used in establishing the stochastic estimates in
\cite{LOTT21}, the estimates here are of purely deterministic nature.


\section{Warm up: the restricted model}\label{sec:restrictedmodel}

\subsection{First algebraic approach: towards the model}\label{subsec:algebraicapproach}
\mbox{}

We follow the philosophy of rough paths in the sense that we seek a theory 
independent on the specific function $a$, just depending on how $a$ enters the equation. 
This lifting to a different level of description has the advantage of linearizing the problem, 
at the expense of making it (more) infinite dimensional. 
Here, it leads to the PDE analogue $\Pi$ of the iterated integrals in rough path theory.
Incidentally, this type of lifting of a nonlinear setting to an infinite-dimensional linear setting is ubiquitous in mathematics, where we think in particular of Young measures.

\medskip

For the sake of this discussion, we assume that $a$ is analytic, so that
\begin{equation}\label{rm5}
\z_{k}[a]:=\frac{1}{k!}\frac{d^{k}a}{dv^{k}}(0),\;k\geq 0
\end{equation}
provide coordinates on the space of all $a$'s. 
The function $a$ can be recovered\footnote{for $v$ within the domain of convergence} via
\begin{equation}\label{rm1}
	a(v)=\sum_{k\geq 0}\z_{k}[a] v^{k}.
\end{equation}
Let $u$ be a space-time periodic function and $\lambda$ a constant satisfying
\begin{equation}\label{rm2}
\left\{\begin{array}{l}
(\partial_{2}-\partial_1^2)u+\lambda=a(u)\partial_{1}^{2}u+\xi,\\
\fint_{[0,1)^2} u=0,
\end{array}
\right.
\end{equation}
where the second line denotes the average over the unit cell $[0,1)^2$.
By standard small data theory for nonlinear parabolic equations\footnote{which we only
appeal to in this motivation}, 
\eqref{rm2} has a unique small solution 
$(u,\lambda)$ for sufficiently small and regular $\xi$. 
In view of \eqref{rm1}, we thus may consider this $(u,\lambda)$ as a
(smooth) function of the variables $\{\z_{k}\}_{k\geq 0}$, 
take partial derivatives with respect to these variables, and set them to zero. 
In other words, given a multi-index\footnote{which just
means that $\beta(k)=0$ for all but finitely many $k$'s} $\beta=(\beta(k))_{k\geq 0}$, we apply
\begin{equation}\label{rm4bis}
\partial^{\beta}:=\frac{1}{\beta!}\prod_{k\geq 0}
\Big(\frac{\partial}{\partial \z_{k}\big|_{\z_{k}=0}}\Big)^{\beta(k)}
\end{equation}
to \eqref{rm2}, which we rewrite as
\begin{equation}\label{rm3}
(\partial_{2}-\partial_{1}^{2})u+\lambda=\sum_{k\geq 0}\z_{k}[a] u^{k}\partial_{1}^{2}u+\xi.
\end{equation} 
By Leibniz' rule, we obtain\footnote{Note that for fixed $\beta$ the sums on the r.~h.~s. are effectively finite:
The condition $e_{k}+\beta_1 +... + \beta_{k+1}=\beta$ is non-empty only if
$\beta(k)\neq 0$, which is true for only finitely many $k$, and only if $\beta_j\le\beta$,
which again is only true for finitely many $\beta_j$'s.}
\begin{align}\label{rm3bis}
(\partial_{2}-\partial_{1}^{2})\partial^{\beta}u+\partial^{\beta}\lambda&
=\sum_{k\geq 0}\sum_{\beta_0 + ... + \beta_{k+1}=\beta} 
\partial^{\beta_{0}}\z_{k} \partial^{\beta_{1}}u\cdots \partial^{\beta_{k}}u\partial_{1}^{2} 
\partial^{\beta_{k+1}}u+\delta_{\beta}^{0}\xi\nonumber\\
&=\sum_{k\geq 0}\sum_{e_{k}+\beta_1 + ... + \beta_{k+1}=\beta} \partial^{\beta_{1}}u\cdots \partial^{\beta_{k}}u\partial_{1}^{2}\partial^{\beta_{k+1}}u+\delta_{\beta}^{0}\xi.
\end{align}
Note that the r.~h.~s. involves $\partial^{\beta'}u$
only if $\beta'$ has a length $\sum_{k\ge 0}\beta'(k)$ strictly less than the one of $\beta$.
Hence \eqref{rm3bis} amounts to a hierarchy of linear PDEs,
which together with $\fint_{[0,1)^2}\partial^\beta u$ characterize
the family $\{\partial^\beta u\}_{\beta}$.
Hence we may, by induction in the length of $\beta$, 
construct a tuple $(\Pi_\beta,C_\beta)$ of $(\mbox{periodic function},\mbox{constant})$ such that
\begin{equation}\label{rm4}
\left\{\begin{array}{l}
(\partial_{2}-\partial_{1}^{2})\Pi_{\beta}+C_{\beta}=\sum_{k\geq 0}\sum_{e_{k}+\beta_1 +...+\beta_{k+1} =\beta}\Pi_{\beta_{1}}\cdots\Pi_{\beta_{k}}\partial_{1}^{2}\Pi_{\beta_{k+1}}+\delta_{\beta}^{0}\xi,\\
\Pi_{\beta}\text{ periodic, }\fint_{[0,1)^2}\Pi_{\beta}=0.
\end{array}
\right.
\end{equation}
This will be rigorously done in Subsection \ref{subsec:statmodel} in a more general context.

\medskip

There is a more compact way of writing \eqref{rm4} which involves an algebraic structure
useful in the sequel. Indeed, since $(\Pi_\beta(x),C_\beta)$ 
was obtained by applying \eqref{rm4bis} to $(u(x),\lambda)$, we may interpret
$(\Pi_\beta(x),C_\beta)_{\beta}$ as coefficients of a power series in 
the variables $\{\mathsf{z}_k\}_{k\ge 0}$. This power series is only formal since
the series $\sum_{\beta}\Pi_\beta(x)\mathsf{z}^\beta$, where $\mathsf{z}_\beta$ denotes the monomial
$\prod_{k\ge 0}\mathsf{z}_k^{\beta(k)}$, may not converge, even when evaluated for 
polynomial $a$.
However, the space $\mathbb{R}[[\mathsf{z}_k]]$ of formal power series in the 
(infinite set of) variables $\{\mathsf{z}_k\}_{k\ge 0}$ with coefficients in $\mathbb{R}$,
makes perfect sense: Its elements $\pi\in\mathbb{R}[[\mathsf{z}_k]]$ are the (infinite) sequences
$(\pi_{\beta})_{\beta}$ in $\mathbb{R}$. 
The linear space $\mathbb{R}[[\mathsf{z}_k]]$ of formal power series is naturally a 
commutative unital algebra, where the product of two elements $\pi$, $\pi' $ is given by
\begin{equation}\label{rmmult}
(\pi\pi')_{\beta}=\sum_{\gamma+\gamma'=\beta}\pi_{\gamma}\pi'_{\gamma'}
\end{equation}
and the neutral element, which we denote by $\dum{1}$, is given by $\dum{1}_{\beta}=\delta_{\beta}^{0}$. 

\medskip

Equipped with this notation and structure, we may regard $\Pi$
as a periodic function of $x$ with values in $\mathbb{R}[[\mathsf{z}_k]]$ and
$C$ as an element of $\mathbb{R}[[\mathsf{z}_k]]$. Moreover, in view of \eqref{rmmult} we may
rewrite \eqref{rm4} more compactly as
\begin{equation}\label{rm6}
\left\{\begin{array}{l}
(\partial_{2}-\partial_{1}^{2})\Pi+C=\sum_{k\geq 0}\dum{z}_{k}\Pi^{k}\partial_{1}^{2}\Pi+\xi\dum{1},\\
\Pi\;\mbox{periodic,}\;\fint_{[0,1)^2}\Pi=0.
\end{array}
\right.
\end{equation}

\medskip

While \eqref{rm6} will be made completely rigorous, we may only formally return
to the general solution $(u,\lambda)$. While it is tempting to interpret 
$\pi\in\R[[\z_k]]$ as a nonlinear but algebraic functional on the space of $a$'s via
%
%
\begin{align}\label{co09}
\pi[a]=\sum_{\beta}\pi_\beta\z^\beta[a],
\end{align}
this is only formal, because even for a polynomial $a$, this series does not
converge for all $\pi\in\mathbb{R}[[\mathsf{z}_k]]$. However, it is convenient to
think along the lines of \eqref{co09} since in view of the derivation of $\Pi$ we
formally have
\begin{align}\label{co14}
(u,\lambda)=(\Pi[a],C[a])\;\;\mbox{is a solution of}\;\eqref{rm2}.
\end{align}

\medskip


\subsection{Structure for shifts: towards the structure group}\label{subsec:shifts}
\mbox{}

An emanation of the specific structure of the equation \eqref{eq} is the following
invariance: If $u$ is a solution then $\tilde u=u-\pi^{(\0)}$, 
for some shift $\pi^{(\0)}\in\mathbb{R}$, is a solution
of the same equation but for the shifted nonlinearity $\tilde a=a(\cdot+\pi^{(\0)})$. 
We will use this invariance to ``tilt'' space-time averages 
in \eqref{rm2} by purely algebraic means. 
For this, we need to generalize the shift 
$\pi^{(\0)}\in\mathbb{R}$ 
to an element\footnote{ We identify $\pi^{(\0)}\in\mathbb{R}$
with $\pi^{(\0)}\mathsf{1}\in\mathbb{R}[[\mathsf{z}_k]]$.} 
$\pi^{(\0)}\in\mathbb{R}[[\mathsf{z}_k]]$. Since in view of \eqref{co09},
$\pi^{(\0)}\in\mathbb{R}[[\mathsf{z}_k]]$ can be (formally) interpreted as an $a$-dependent shift,
it formally gives rise to the (nonlinear) map on the space of $a$'s
\begin{align}\label{co02}
a\mapsto \tilde a:=a(\cdot +\pi^{(\0)}[a]).
\end{align}
This map (formally) lifts to an algebra morphism\footnote{ The notation $^*$ will become clear in Section \ref{sec:fullmodel}.} $\Gamma^*$ of $\mathbb{R}[[\mathsf{z}_k]]$ via 
\begin{align}\label{co01}
\Gamma^*\pi [a]:=\pi[\tilde a]\quad\mbox{with}\;\tilde a\;\mbox{defined in}\;\eqref{co02}.
\end{align}
\ignore{
As a consequence of Taylor's formula for a polynomial $a$, 
\eqref{co01} translates into
\begin{equation}\label{gamma1}
\Gamma^* \dum{z}_{k}=\sum_{l\geq 0}\left(\begin{matrix}
k+l\\k
\end{matrix}\right)(\pi^{(\0)})^{l}\dum{z}_{k+l}.
\end{equation}
Indeed, 
\begin{align*}
a(v+\pi^{(\0)})&\underset{\eqref{rm2}}{=}\sum_{k\geq 0}\frac{1}{k!}\frac{d^{k}a}{dv^{k}}(0)\sum_{l=0}^{k}\left(\begin{matrix}k\\l\end{matrix}\right)v^{l}(\pi^{(\0)})^{k-l}\\
&=\sum_{l\geq 0}\Big(\sum_{k\geq 0}\left(\begin{matrix}l+k\\l\end{matrix}\right)\frac{1}{(k+l)!}\frac{d^{k+l}a}{dv^{k+l}}(\pi^{(\0)})^{k}\Big)v^{l}.
\end{align*}
We note that the sum in \eqref{gamma1} is effectively finite, meaning that when considering a
component $\beta$, all but finitely many terms vanish. 
}
Still formally, as we shall presently argue,
$\Gamma^*$ provides a tilt of the space-time average in \eqref{rm6}
in the sense that
\begin{align}\label{cco2}
(\tilde\Pi,\tilde C):=(\Gamma^* \Pi+\pi^{(\0)},\Gamma^* C)
\end{align}
satisfies
\begin{align}\label{co07}
\left\{\begin{array}{l}(\partial_2-\partial_1^2)\tilde\Pi+\tilde C=\sum_{k\ge 0}\mathsf{z}_k\tilde\Pi^k\partial_1^2\tilde\Pi+\xi\mathsf{1},\\
\fint_{[0,1)^2}\tilde\Pi=\pi^{(\0)}.
\end{array}\right.
\end{align}
On a formal level, this is almost tautological.
Indeed, the second line in \eqref{co07} follows from the second line of \eqref{rm6}
and the fact that $\fint_{[0,1)^2}$ commutes with the linear $\Gamma^*$. 
In view of the definition \eqref{co01} of $\Gamma^*$
and the formal principle \eqref{co14}, $(\Gamma^*\Pi[a],\Gamma^*C[a])$ 
is the solution of \eqref{rm2} with $a$ replaced by $\tilde a$.
By definition \eqref{co02} of $\tilde a$ and the above-mentioned invariance,
$(u,\lambda):=(\Gamma^*\Pi[a]+\pi^{({\bf 0})},\Gamma^*C[a])$ satisfies the first
line of \eqref{rm2} (for the original $a$). By definition \eqref{cco2} and
the formal principle \eqref{co14}, $(\tilde\Pi,\tilde C)$ thus satisfies
the first line of \eqref{co07}.
In the more general context of Subsection \ref{subsec:centeredmodel}, we will provide a rigorous argument 
for this transformation property.

\medskip

We now characterize more explicitly the $\Gamma^*$ (formally) defined in \eqref{co01} for given 
$\pi^{(\0)}\in\mathbb{R}[[\mathsf{z}_k]]$. To this purpose, we replace $\pi^{(\0)}$
by $t\pi^{(\0)}$ in \eqref{co02} and consider the corresponding $\Gamma_t^*$.
It follows immediately from \eqref{co02} and \eqref{co01} that $\Gamma_t^*$ satisfies
the initial value problem
\begin{align}\label{gamma2}
\frac{d}{dt}\Gamma_t^*=\pi^{(\0)}\Gamma_t^*D^{(\0)}\quad\mbox{and}\quad\Gamma_{t=0}^*=\textnormal{id},
\end{align}
where $D^{(\0)}$ is the infinitesimal generator of shifts of $u$-space, that is,
\begin{align}\label{co03}
D^{(\0)}\pi[a]=\frac{d}{dv}_{|v=0}\pi[a(\cdot + v)].
\end{align}
From \eqref{gamma2}, we (formally) read off the identity
\begin{align}\label{exp1}
\Gamma^*=\sum_{l\ge 0}\frac{1}{l!}(\pi^{(\0)})^l(D^{(\0)})^l.
\end{align}
We stress that since multiplication with $\pi^{(\0)}$ and applying
$D^{(\0)}$ do not commute, $\Gamma^*$ is not the matrix exponential of $\pi^{(\0)}D^{(\0)}$.
In fact, \eqref{gamma2} is not the flow arising from an ODE on $\mathbb{R}[[\mathsf{z}_k]]$; 
in particular, the usual group property does not hold
-- it would in fact contradict the group property \eqref{gp1} derived below.
One should rather see the ODE in \eqref{gamma2} as providing a flow
on the (non-linear) space of algebra automorphisms ${\rm Alg}(\mathbb{R}[[\mathsf{z}_k]])$,
by considering more general initial conditions than the identity.
In Subsection \ref{subsec:structuregroup}, formula \eqref{exp1} will serve as a rigorous definition of $\Gamma^*$.

\medskip

As follows immediately from \eqref{rm5} and \eqref{co03}, on the coordinates $\mathsf{z}_k$,
$D^{(\0)}$ acts as $D^{(\0)}\mathsf{z}_k=(k+1)\mathsf{z}_{k+1}$. Since by \eqref{co03},
$D^{(\0)}$ is a derivation on the algebra $\mathbb{R}[[\mathsf{z}_k]]$, it thus must be of the form
\begin{align}\label{der1}
D^{(\0)}=\sum_{k\ge 0}(k+1)\mathsf{z}_{k+1}\partial_{\mathsf{z}_k}.
\end{align}
%
The appearance of a derivation is
natural, since its defining properties $D\pi\pi'=(D\pi)\pi'+\pi(D\pi')$ 
and $D\mathsf{1}=0$ are the infinitesimal versions of the algebra morphism properties 
$\Gamma^*\pi\pi'=(\Gamma^*\pi)(\Gamma^*\pi')$ and $\Gamma^*\mathsf{1}=\mathsf{1}$,
respectively. 

\medskip

From \eqref{der1} it is not immediate that $D^{(\0)}$ is well-defined
as an element of ${\rm End}(\mathbb{R}[[\mathsf{z}_k]])$. A sufficient condition for a map $M$ to be well-defined as an element of ${\rm End}(\mathbb{R}[[\mathsf{z}_k]])$ is that in terms of its matrix representation\footnote{The matrix representation
	$M_\beta^\gamma$ is defined through $(M\pi)_\beta
	=\sum_{\gamma}M_\beta^\gamma\pi_\gamma$.} 
the set $\{\gamma\;|\;M_\beta^\gamma\not=0\}$ is finite for every $\beta$; this is equivalent to $M$ being the transpose of an element of ${\rm End}(\mathbb{R}[\z_k])$.
As can be seen from \eqref{der1}, the matrix representation of $D^{(\0)}$ is given by
\begin{align}\label{der2}
(D^{(\0)})^\gamma_\beta=\sum_{k\ge0}
\left\{\begin{array}{ll}
(k+1)\gamma(k)&\mbox{if}\;\beta+e_k=\gamma+e_{k+1},\\
0&\mbox{otherwise}.
\end{array}\right.
\end{align}
Now indeed, for given $\beta$, there are only finitely many $k\ge 0$ such that 
$\beta+e_k=\gamma+e_{k+1}$ for some $\gamma$ (namely those $k$'s for which $\beta(k+1)\not=0$)
and thus only finitely many $\gamma$'s for which $(D^{(\0)})^\gamma_\beta\not=0$.

\medskip

In a related spirit, it is not clear that the sum \eqref{exp1} is effectively finite,
meaning that it is finite when evaluated in a matrix element $(\Gamma^*)_\beta^\gamma$;
it is also not clear that the resulting $\{\gamma\;|\;(\Gamma^*)_\beta^\gamma\not=0\}$ is finite 
for every $\beta$, so that \eqref{exp1} defines an element in 
${\rm End}(\mathbb{R}[[\mathsf{z}_k]])$.
In fact, both will be established in Lemma \ref{lemts2} in a more general context. 
The latter finiteness property is related to the triangularity property\footnote{ However, \eqref{co03bis} does not imply finiteness, due to the $e_0$-component; this will become important later, see e.~g. Subsection \ref{subsec:strategy}.}
\begin{align}\label{co03bis}
(\Gamma^*-\textnormal{id})_\beta^\gamma\neq 0\implies\;|\gamma|_s<|\beta|_s,
\end{align}
where $|\beta|_s:=\sum_{k\ge0}k\beta(k)$ is a scaled length of the multi-index $\beta$.
Indeed, in view of \eqref{gamma2}, \eqref{co03bis} follows (formally) from
\begin{align}\label{rm20}
(D^{(\0)})_\beta^\gamma\neq 0\implies\;|\gamma|_s<|\beta|_s,
\end{align}
which follows from \eqref{der2}. However we give a separate argument: \eqref{rm20} is the matrix version of the mapping property
\begin{align}\label{co05}
\pi_\gamma=0\;\mbox{for}\;|\gamma|_s\le k
\quad\Longrightarrow\quad(D^{(\0)}\pi)_\beta=0\;\mbox{for}\;|\beta|_s\le k+1,\;\mbox{for any }k\ge 0,
\end{align}
which we claim to be true. Indeed, returning to our interpretation \eqref{co09},
the component-wise defined
subspaces appearing in \eqref{co05} are characterized by a scaling property:
\begin{align*}
\pi_\beta=0\;\mbox{for}\;|\beta|_s\le k
\quad\Longleftrightarrow\quad \lim_{\lambda\uparrow\infty}\lambda^k\pi_\lambda=0,
\end{align*}
where we write $\pi_\lambda[a]:=\pi[a(\frac{\cdot}{\lambda})]$.
Since in view of \eqref{co03} we have $D^{(\0)}\pi_\lambda=\lambda(D^{(\0)}\pi)_\lambda$,
the mapping property \eqref{co05} follows.

\medskip

We note that \eqref{co03bis} does not imply the finiteness of 
$\{\gamma\;|\;(\Gamma^*)_\beta^\gamma\not=0\}$ since
$|\gamma|_s$ does not control the component $\gamma(0)$; we will give
a rigorous argument for this in the more general context of Subsection \ref{subsec:structuregroup}.
However, once finiteness is clear, \eqref{co03bis} does imply the invertibility of $\Gamma^*$, 
that is, the unique solvability of the equation $\pi'=\Gamma^*\pi$.
Indeed, in view of \eqref{co03bis} we may write the system of linear equations as
$\pi'_\beta=\pi_\beta
+\sum_{l=0}^{|\beta|_s-1}\sum_{\gamma:|\gamma|_s=l}(\Gamma^*)_\beta^\gamma\pi_\gamma$,
so that solvability follows by induction over $|\beta|_s$.

\medskip

Hence $\Gamma^*$ is an algebra automorphism of $\mathbb{R}[[\mathsf{z}_k]]$.
In fact, the set of $\Gamma^*$'s generated by $\pi^{(\0)}\in\mathbb{R}[[\mathsf{z}_k]]$
through \eqref{co01} or rather \eqref{exp1} form a subgroup
of the automorphism group.
It remains to argue that this set is closed under composition. Indeed, this is
true in the following specific sense: Suppose $\pi^{(\0)}$ and $\pi'^{(\0)}$ generate $\Gamma^*$ and $\Gamma'^*$, respectively. 
Then 
\begin{equation}\label{gp1}
\pi^{(\0)}+\Gamma^*\pi'^{(\0)}\text{ generates }\Gamma^*\Gamma'^*. 
\end{equation}
On the level of the formal definition \eqref{co01}, namely
%
\begin{align*}
\Gamma^*\pi[a]=\pi[a(\cdot +\pi^{(\0)}[a])],
\end{align*}
this is almost tautological.
Momentarily introducing the notation $\pi':=\Gamma'^*\pi$ and 
$a':=a(\cdot+\pi^{(\0)}[a])$, this characterization yields the desired identity:
\begin{align*}
\Gamma^*\Gamma'^*\pi[a]=\Gamma^*\pi'[a]=\pi'[a']=\pi[a'(\cdot + \pi'^{(\0)}[a'])]=\pi[a(\cdot + \pi^{(\0)}[a]+ \Gamma^* \pi'^{(\0)}[a])].
\end{align*}
Alike before, we postpone the rigorous argument to Subsection \ref{subsec:structuregroup}.


\section{The full model}\label{sec:fullmodel}
Our goal is to capture, at least formally, the entire solution 
manifold of \eqref{eq}. While the restricted model $\Pi$ from the previous subsection 
formally captures all periodic solutions of \eqref{eq}, it clearly does not capture
all local behavior, as can be seen by considering the special case of $a\equiv 0$:
In this case, we may add to the periodic solution $v$ any ``caloric'' function $p$,
i.~e.~a function that solves $(\partial_2-\partial_1^2)p=0$ and thus automatically is
analytic on $\mathbb{R}^2$ -- we restrict to polynomials
in line with the algebraic approach we embarked on
in Section \ref{sec:restrictedmodel}. In view of the Cauchy-Kovalevskaya theorem, 
it is plausible that also in the presence of an analytic nonlinearity $a$,
the solution manifold (locally) can be parameterized by caloric $p$'s; loosely speaking, we
seek to deform the affine solution manifold at $a\equiv 0$. For this purpose, it is convenient 
to give up the restriction that $q$ is caloric, and, as a consequence, to 
relax\footnote{as it turns out, this relaxation is ultimately immaterial} 
\eqref{eq} to hold only up to a polynomial $q$.

\medskip

We will proceed in several steps: In Subsection \ref{subsec:statmodel}, following the approach
of Subsection \ref{subsec:algebraicapproach}, we construct what we assimilate to
the pre-model\footnote{We adopt the language of \cite[Definition 2.20]{Che22}.} Alongside,
we introduce the model space (or rather its algebraic dual $\T^*$) and the set of homogeneities $\mathsf{A}$. 
In Subsection \ref{subsec:centeredmodel}, we construct the centered model
$\Pi_x$, our main object of interest. We show that the centered model arises
from the pre-model by a purely algebraic transformation of the model space. 
In Subsection \ref{subsec:structuregroup}, we construct this transformation and show that it gives
rise to the structure group $\mathsf{G}$ (rather its pointwise dual $\mathsf{G}^*$), and to the re-expansion maps $\Gamma_{xy}^*$.

\subsection{The pre-model}\label{subsec:statmodel}
\mbox{}

Let us motivate our Ansatz for the pre-model $\mathbf{\Pi}$ as we did for $\Pi$
in Subsection \ref{subsec:algebraicapproach}.
There we formally obtained the model by taking derivatives of the periodic solution $u$
of vanishing space-time average, see \eqref{rm2}, with respect to the nonlinearity $a$. 
We now have in mind a non-periodic solution $u$. 
In order to retain uniqueness, we consider $u$ to be of the following form:
We first double space-time variables by considering polynomials in a second space-time variable 
$y\in\mathbb{R}^2$ with coefficients that are periodic functions of space-time\footnote{ This is the tensor product of periodic functions with polynomials.}; 
we then restrict such a function to the diagonal. 
We extend \eqref{rm2} to 
\begin{align}\label{co10}
\left\{\begin{array}{l}
(D_2-D_1^2)u+q=a(u)D_1^2 u + \xi,\\
\fint_{[0,1)^2}u(y)=p(y),
\end{array}\right.
\end{align}
where $u$ is of the above form, $p$ and $q$ are polynomials, the operator $D_i$ now acts on both the periodic and the polynomial variables\footnote{ i.~e. $D_i = \partial_i + \frac{\partial}{\partial y_i}$}, and $\fint_{[0,1)^2}u(y)$ denotes the average of $u$ over the periodic variable with the polynomial variable evaluated at $y$.
As for \eqref{rm2}, given $(a,p)$, 
one expects existence and uniqueness of $(u,q)$ in a small data setting.
This morally justifies considering derivatives of $(u,q)$ with respect to $(a,p)$. In view of \eqref{co07},
we only need to consider $p$ up to a constant; arbitrarily fixing an origin, we thus
restrict to $p$'s with $p(0)=0$. Hence the analogue of \eqref{rm5} are the variables
\begin{align}\label{co13}
\z_k[a,p] := \frac{1}{k!}\frac{d^k a}{dv^k} (0),\;k\geq0,\quad\quad \z_{\bf n}[a,p]:=\frac{1}{\n!}\frac{\partial^{\bf n}p}{\partial y^{\bf n}}(0),\; {\bf n}\in\mathbb{N}_0^2\setminus\{(0,0)\},
\end{align}
where $\frac{\partial^{\bf n}}{\partial y^{\bf n}}$ 
$:=(\frac{\partial}{\partial y_1})^{n_1}(\frac{\partial}{\partial y_2})^{n_2}$.
We will henceforth write ${\bf n}\not=\0$ for ${\bf n}\in\mathbb{N}_0^2\setminus\{(0,0)\}$. 
Multi-indices $\beta$ over $k\geq 0$ and $\n\neq \0$
measure the frequency of the variables $\{\z_k,\z_\n\}$, so that we can write monomials
\begin{equation}\label{co15}
\z^\beta := \prod_{k\ge 0,\n\not=\0} \z_k^{\beta(k)} \z_{\bf n}^{\beta (\n)},
\end{equation}
and derivatives with respect to $\{\z_k,\z_\n\}$. Similarly to
\eqref{rm4} in Subsection \ref{subsec:algebraicapproach} we obtain \eqref{fm3} below,
where now $\mathbf{\Pi}_\beta$ is a polynomial with periodic coefficients, and $P_\beta$
is a polynomial. Introducing $\mathbb{R}[[\mathsf{z}_k,\mathsf{z}_{\bf n}]]$,
this can be compactly written: \eqref{rm6} extends to \eqref{fm2} below.

\medskip

Extending \eqref{co09}, we formally
think of an element $\pi$ of $\mathbb{R}[[\mathsf{z}_k,\mathsf{z}_{\bf n}]]$ 
as a functional on the space of pairs $(a,p)$ of nonlinearities and polynomials:
\begin{align}\label{topp.13}
\pi[a,p]=\sum_{\beta}\pi_\beta\prod_{k\ge0,{\bf n}\not=\0}
\big(\frac{1}{k!}\frac{d^k a}{dv^k}(0)\big)^{\beta(k)}
\big(\frac{1}{{\bf n}!}\frac{\partial^{\bf n}p}{\partial y^{\bf n}}(0)\big)^{\beta({\bf n})}.
\end{align}
This formally implies
\begin{align}\label{co12}
(u,q)=(\mathbf{\Pi}[a,p],P[a,p])
\;\;\mbox{solves}\;\eqref{co10}.
\end{align}

\medskip

Summing up, and moving beyond this heuristic derivation, there are two
effects of adjoining space-time polynomials: 
The first one is to add a polynomial space-time variable $y$; 
the second one is to add the variables $\{\mathsf{z}_{\bf n}\}_{\bf n\not=0}$. 
The first effect means that we replace the space of periodic functions by 
the space of functions that are diagonal restrictions of polynomials in $y$ 
with coefficients being periodic functions.  
The second effect means that we replace the algebra $\mathbb{R}[[\mathsf{z}_k]]$ by
the algebra $\mathbb{R}[[\mathsf{z}_k,\mathsf{z}_{\bf n}]]$.

\begin{proposition}\label{defstatmodel}
	There exits a unique pair $(\mathbf{\Pi},P):\R^{2}\times\R^{2}\to \R[[\dum{z}_{k},\dum{z}_{\n}]] $ such that for every multi-index $\beta$, $\mathbf{\Pi}_{\beta}(x,y)$ is a polynomial in $y$ the coefficients of which are periodic functions in $x$, $P_\beta$ is a polynomial, and they satisfy
	\begin{equation}\label{fm2}
	\left\{\begin{array}{l}
	\left(D_{2}-D_{1}^{2}\right)\mathbf{\Pi} + P=\mathbf{\Pi}^-,\\
	\mathbf{\Pi}^-:=\sum_{k\geq 0}\dum{z}_{k}\mathbf{\Pi}^{k}D_{1}^{2}\mathbf{\Pi}+\xi\dum{1},\\
	\fint_{[0,1)^2}\mathbf{\Pi}(\cdot,y)=\sum_{\n\not=\0}y^{\n}\dum{z}_{\n}.
	\end{array}
	\right.
	\end{equation}


	We call $\mathbf{\Pi}$ \textbf{pre-model}. It satisfies, in addition, the following property: 
	\begin{equation}\label{pop12}
		\mathbf{\Pi}_{\beta}\not\equiv 0 \,\implies\,\left\{\begin{array}{l}
			\beta=e_{\n}\text{ for some }\n\not=\0,\text{ or}\\
			\left[\beta\right]:= \sum_{k\geq 0}k\beta(k)-\sum_{\n\neq \0}\beta(\n)\geq 0.
		\end{array}\right.
	\end{equation}
	Moreover, 
	\begin{align}
	\mathbf{\Pi}_{e_{\n}}(\cdot,y)&=y^{\n}\text{ for all }\n\neq \0,\label{sm11}\\
	P(y) &= -(\frac{\partial}{\partial y_2} - \frac{\partial^2}{\partial y_1^2})
\sum_{\n\neq \0} 
y^{\n} \z_\n + \fint_{[0,1)^2}\mathbf{\Pi}^-(y),\label{jan01}\\
	\deg \mathbf{\Pi}_\beta &\leq |\beta|_p := \sum_{\n\neq \0} |\n| \beta(\n),\label{jan02}
	\end{align}	
	where\footnote{ This choice of the scaled length of $\n$ reflects the parabolic scaling. In general, one chooses a scale length according to the scaling dictated by the differential operator.}
	\begin{equation}\label{polyn}
		|\n|:= n_1 + 2 n_2
	\end{equation}
	and $\deg \mathbf{\Pi}_\beta$ denotes the parabolic polynomial degree of $\mathbf{\Pi}_\beta$.
\end{proposition}
%
We will call the multi-indices that satisfy either of the conditions in \eqref{pop12} \textbf{populated},
a nomenclature that refers to the fact that, by Proposition \ref{defstatmodel}, 
these multi-indices are the ones that potentially give non-zero contributions to $\mathbf{\Pi}$, 
i.~e. \textit{populate} the model. In addition, we will call multi-indices of the form of the first line of the r.~h.~s. of
\eqref{pop12} \textbf{purely polynomial}. 
The model $\mathbf{\Pi}$ has values in the corresponding
subspace of $\R[[\dum{z}_{k},\dum{z}_{\n}]]$.
\begin{definition}\label{def:modspace}
	We define the space $\dum{T}^*\subset\R[[\dum{z}_{k},\dum{z}_{\n}]]$ as
	\begin{equation*}
\dum{T}^*:=\left\{\pi\in \R[[\dum{z}_{k},\dum{z}_{\n}]]\;| 
\; \pi_{\beta}\neq 0\implies \beta\text{ is populated, i.~e.~satisfies
\eqref{pop12}}\right\}.
	\end{equation*}
\end{definition}
The population conditions \eqref{pop12} induce a decomposition $\T^* = \bar{\T}^* \oplus \tilde{\T}^*$, where
\begin{align*}
	\bar{\T}^*\hspace*{-4pt}&:=\hspace*{-3pt} \{ \pi\in \R[[\dum{z}_{k},\dum{z}_{\n}]]\;\hspace*{-1pt}|\hspace*{-1pt} 
\; \pi_{\beta}\neq 0 \hspace*{-0.15cm}\implies \hspace*{-0.15cm} \beta\text{ is purely polynomial, i.~e.~
satisfies the first line of }\eqref{pop12} \},\\
	\tilde{\T}^*\hspace*{-4pt}&:=\hspace*{-3pt} \{\pi\in \R[[\dum{z}_{k},\dum{z}_{\n}]]\;\hspace*{-1pt}|\hspace*{-1pt} \; \pi_{\beta}\neq 0\hspace*{-0.15cm}\implies \hspace*{-0.15cm}\beta\text{ satisfies the second line of }\eqref{pop12}\}.
\end{align*}
The notation $\T^*$ is meant to connect to the theory of regularity structures formulated 
in \cite{Ha14,Ha16}.
Indeed, $\T^*$ can be identified with the algebraic dual of the \textbf{model space}\footnote{ An element of our space $\tilde{\T}$ actually encodes both the rooted and the planted elements in the model space of \cite{Ha14,Ha16}. Moreover, the polynomial sector $\bar{\T}$ as defined here does not contain constants. 
	As a consequence, the precise identification of model space in  \cite{Ha14,Ha16} should be 
	$\R\oplus \T\oplus \tilde{\T}$. We refer to 
	\cite[Subsection 5.3]{LOT21} for details. 
	\label{foot01}} $\T$, 
which itself contains the \textbf{polynomial sector} $\mathsf{\bar T}$. 
In fact, $\mathsf{\bar T}^*$ can be identified with the (algebraic) dual
of $\mathbb{R}[y_1,y_2]/\mathbb{R}$, the space of polynomials in $y\in\mathbb{R}^2$ 
with constants factored out, via
\begin{align}\label{cw03}
	\pi[p]=\sum_{{\bf n}\not=\0}\pi_{e_{\bf n}}\frac{1}{{\bf n}!}
	\frac{\partial^{\bf n}p}{\partial y^{\bf n}}(0),
\end{align}
which is in line with \eqref{topp.13}. Up to duality
and constants, this agrees with \cite[Assumption 3.20]{Ha16}.
In particular, the multiplication in the algebra $\mathbb{R}[[\mathsf{z}_{\bf n}]]$
is unrelated to the multiplication of polynomials in $y\in\mathbb{R}^2$, as is highlighted by
$\mathsf{z}_{\bf n}\mathsf{z}_{\bf n'}\not=\mathsf{z}_{{\bf n}+{\bf n'}}$.
As a consequence, $\mathsf{T}^*$ (as opposed to $\mathsf{\tilde T}^*$) 
is not a sub-algebra of 
$\mathbb{R}[[\mathsf{z}_k,\mathsf{z}_{\bf n}]]$. 
While $\mathbf{\Pi}$ is a $\mathsf{T}^*$-valued
function, $\mathbf{\Pi}^-$ and $P$ have values in a larger subspace of 
$\mathbb{R}[[\mathsf{z}_k,\mathsf{z}_{\bf n}]]$:
\begin{remark}\label{rem:new}
If $(\mathbf{\Pi}^-_{\beta},P_\beta)\not\equiv 0$,
$\beta$ is populated or of the form
\begin{equation}\label{pc1}
\beta = e_k + e_{\n_1} +... + e_{\n_{k+1}}
\end{equation}
with $k\ge 0$.
For a $\beta$ of the form of \eqref{pc1}, 
$\mathbf{\Pi}^-_{\beta}$ is a polynomial, and we have 
$[\beta] = -1$. 
\end{remark}
The proof of Proposition \ref{defstatmodel} is based on the following well-posedness result.
\begin{lemma}\label{lem:wp1}
	For all $(f,q)$ polynomial in $y$ with coefficients that are (smooth periodic functions, scalars), there exists a unique pair $(U,P)$ with the same features and such that
	\begin{equation}\label{staronp9}
	\left\{\begin{array}{l}
	\left(D_{2}-D_{1}^{2}\right)U+P=f,\\
	\fint_{[0,1)^2} U=q.\end{array}\right.
	\end{equation}  
	Moreover, $\deg (U-q)\leq\deg f$ and $
	\deg (P-\fint_{[0,1)^2} f) \leq\deg q-2$.
	\begin{proof}
		Let us first take the average over the periodic variable of the equation, i.~e.
		\begin{equation*}
		\left(\frac{\partial}{\partial y_{2}}-\frac{\partial^{2}}{\partial y_{1}^{2}}\right)\fint_{[0,1)^2} U+P=\fint_{[0,1)^2} f.
		\end{equation*}
		Then, since $\fint_{[0,1)^2} U=q$, $P$ is uniquely determined by $P=\fint_{[0,1)^2} f-\left(\frac{\partial}{\partial y_{2}}-\frac{\partial^{2}}{\partial y_{1}^{2}}\right) q$ which moreover implies the statement on the polynomial degree.
		Define now $\tilde{U}:=U-q$, $\tilde{f}:=f-\fint_{[0,1)^2} f$, which satisfy
		\begin{equation}\label{pde2}
		\left\{\begin{array}{l}
		\left(D_{2}-D_{1}^{2}\right)\tilde{U}=\tilde{f},\\
		\fint_{[0,1)^2} \tilde{U}=0.\end{array}\right.
		\end{equation}
		We expand both functions
		\begin{equation*}
		\tilde{U}(\cdot,y)=\sum_{\n}\tilde{U}_{\n}y^{\n},\,\,\,\,\tilde{f}(\cdot,y)=\sum_{\n}\tilde{f}_{\n}y^{\n}	
		\end{equation*}
		where there exist $\n_{\tilde{U}}$, $\n_{\tilde{f}}$ such that $|\n_{\tilde{U}}|=\deg\tilde{U}$ and $|\n_{\tilde{f}}|=\deg\tilde{f}$. Note that both $\n_{\tilde{U}}$ and $\n_{\tilde{f}}$ are fixed and finite -- though $\n_{\tilde{U}}$ is still unknown. Then by equating the coefficients we can translate \eqref{pde2} into the family of problems
		\begin{equation}\label{janeq}
		\left\{\begin{array}{l}
		\hspace*{-0.3cm}(\partial_{2}\hspace*{-0.02cm}-\hspace*{-0.02cm}\partial_{1}^{2})\hspace*{-0.02cm}\tilde{U}_{\n}\hspace*{-0.04cm}= \hspace*{-0.03cm}
2(n_{1}\hspace*{-0.05cm}+\hspace*{-0.03cm}1)\partial_{1}\tilde{U}_{\n+(1,0)}\hspace*{-0.03cm}+\hspace*{-0.04cm}
(n_{1}\hspace*{-0.05cm}+\hspace*{-0.03cm}2)(n_{1}\hspace*{-0.05cm}+\hspace*{-0.03cm}1)\tilde{U}_{\n+(2,0)}
\hspace*{-0.06cm}-\hspace*{-0.04cm}(n_{2}\hspace*{-0.05cm}+\hspace*{-0.03cm}1)\tilde{U}_{\n+(0,1)}\hspace*{-0.04cm}+\hspace*{-0.04cm}\tilde{f}_{\n}.\\
		\hspace*{-0.3cm}\fint_{[0,1)^2} \tilde{U_{\n}}=0;\end{array}\right.
		\end{equation}
We solve \eqref{janeq} by Fourier series (recalling that $\tilde f_{\bf n}$
are periodic and smooth), starting at $\n=\max\{\n_{\tilde{U}},\n_{\tilde{f}}\}$ 
and inductively decreasing $|\n|$; this is possible since, by definition of $\tilde f$ 
and by induction hypothesis on $\tilde U$, 
the r.~h.~s. of \eqref{janeq} vanishes under $\fint_{[0,1)^2}$. It remains only
to bound $\deg\tilde{U}$. If $\deg\tilde{U}>\deg\tilde{f}$, then $(\partial_{2}-\partial_{1}^{2})\tilde{U}_{\n_{\tilde{U}}}=0$, $\fint_{[0,1)^2}\tilde{U}_{\n_{\tilde{U}}}=0$, which implies $\tilde{U}_{\n_{\tilde{U}}}\equiv 0$ giving a contradiction. Therefore $\deg\tilde{U}\leq\deg\tilde{f}\leq\deg{f}$, so that the desired bound on the polynomial degree of $U$ holds.
	\end{proof}
\end{lemma}
\medskip

\begin{proof}[Proof of Proposition \ref{defstatmodel}]
Component-wise, \eqref{fm2} takes the form of
\begin{equation}\label{fm3}
\left\{\begin{array}{l}
\left(D_{2}-D_{1}^{2}\right)\mathbf{\Pi}_{\beta}+P_{\beta}=\sum_{k\geq 0}\sum_{e_{k}+\beta_1 +...+\beta_{k+1}=\beta}\mathbf{\Pi}_{\beta_{1}}\cdots\mathbf{\Pi}_{\beta_{k}}D_{1}^{2}\mathbf{\Pi}_{\beta_{k+1}}+\xi\delta_{\beta}^0,\\
\fint_{[0,1)^2}\mathbf{\Pi}_{\beta}(y)=\left\{\begin{array}{ll}
y^{\n} & \text{if }\beta=e_{\n},\\
0 & \text{otherwise.}
\end{array}
\right.
\end{array}
\right.
\end{equation}
As we pointed out in Subsection \ref{subsec:algebraicapproach}, 
for each $\beta$ the r.~h.~s. of the equation is a finite sum in $k$. 
Moreover, the summands only involve multi-indices $\beta_{i}$ of length strictly less 
than the length of $\beta$, for $i=1,...,k+1$. Hence appealing to
Lemma \ref{lem:wp1}, the pair $(\mathbf{\Pi}, P)$ may 
be constructed by induction on the length of $\beta$, and is unique. 
Identity \eqref{jan01} is easily seen to hold by applying $\fint_{[0,1)^2}$
to the first line of \eqref{fm2} and feeding in the last.

\medskip

Let us now turn to the proof of \eqref{pop12}, \eqref{sm11} and \eqref{jan02}. 
We first treat multi-indices $\beta\not=0$ such that $\beta(k)=0$ for all $k\geq 0$.
For such $\beta$'s, the condition $e_{k}+\beta_1 +...+\beta_{k+1}=\beta$	is empty for all $k\geq 0$, and therefore \eqref{fm3} reduces to
		\begin{equation*}
		\left\{\begin{array}{l}
		\left(D_{2}-D_{1}^{2}\right)\mathbf{\Pi}_{\beta}+P_{\beta}=0,\\
		\fint_{[0,1)^2}\mathbf{\Pi}_{\beta}(y)=\left\{\begin{array}{ll}
		y^{\n} & \text{if }\beta=e_{\n},\\
		0 & \text{otherwise.}
		\end{array}
		\right.
		\end{array}
		\right.
		\end{equation*}
By the uniqueness statement in Lemma \ref{lem:wp1}, if $\beta$ is not
purely polynomial, then $\mathbf{\Pi}_{\beta}=0$, whereas if $\beta$ is
purely polynomial we obtain $\mathbf{\Pi}_{e_{\n}}(\cdot,y)=y^{\n}$.
This establishes \eqref{pop12}, \eqref{sm11}, and thus \eqref{jan02}
for this class of multi-indices.
		
		\medskip
		
We now consider all remaining multi-indices. Let us first note that if $\beta$ is of the form 
\eqref{pc1}, the r.~h.~s. of \eqref{fm3} is non-vanishing only for a single $k$ and, 
due to \eqref{sm11} it is a polynomial. By the uniqueness statement in Lemma \ref{lem:wp1}, 
the contribution (which equals its space-time average) is absorbed in $P_\beta$ and 
thus $\mathbf{\Pi}_\beta = 0$; this establishes \eqref{pop12}, 
and trivially \eqref{jan02} for the class of multi-indices of the form \eqref{pc1}. 
For $\beta$ not of this form, we show \eqref{pop12} by induction in the length 
$\sum_{k\ge 0}\beta(k)+\sum_{\bf{n}\not={\bf 0}}\beta(\bf{n})$. 
When the length vanishes, we have $\beta=0$, 
for which all statements follow from Lemma \ref{lem:wp1}.
We consider the r.~h.~s. term in \eqref{fm3}; more precisely let $k\ge 0$ and multi-indices 
$\{\beta_j\}_{j=1,\cdots,k+1}$ be such
		that $e_k+\beta_1 +...+\beta_{k+1}=\beta$. 
The latter relation implies $k +[\beta_1]+...+[\beta_{k+1}]=[\beta]$, cf. \eqref{pop12}. 
Since the length of the $\beta_j$'s is less than the one of $\beta$, we may
apply the induction hypothesis and thus learn that the r.~h.~s. vanishes
unless for all $j=1,...,k+1$, either $[\beta_j]\ge 0$ or $\beta_j$ is purely polynomial, 
in which case $[\beta_j]=-1$. Hence we have $[\beta]\ge 0$ unless all $\beta_j$'s were
purely polynomial, which would imply that $\beta$ is of the form \eqref{pc1}, which we
ruled out. In addition, note that for $\{\beta_j\}_{j=1,\cdots,k+1}$ such
that $e_k+\beta_1 +...+\beta_{k+1}=\beta$ by the induction hypothesis \eqref{jan02}
\begin{equation*}
	\deg\big(\mathbf{\Pi}_{\beta_1} \cdots \mathbf{\Pi}_{\beta_k} D_1^2 \mathbf{\Pi}_{\beta_{k+1}}\big)\leq |\beta_1|_p + ... + |\beta_{k+1}|_p = |\beta|_p,
\end{equation*}
and thus \eqref{jan02} for $\beta$ follows from Lemma \ref{lem:wp1}.
\end{proof}

\medskip

In preparation of Subsection \ref{subsec:structuregroup} we now introduce the notion
of {\bf homogeneity} $|\beta|$ of a multi-index $\beta$,
which will provide a grading of the model space $\mathsf{T}$.
It generalizes the scaled length $|\beta|_s$ that naturally came up in
Subsection \ref{subsec:shifts}, see \eqref{co03bis}. Like $|\beta|_s$ there, 
$|\beta|$ can be motivated by a scaling
argument, which we shall give now. Recall that for $\xi$ we have in mind a noise that
is not only shift invariant in law, but also scale invariant in law\footnote{For the sake of this
discussion we ignore that this rescaling changes the period.}, 
by which we mean that for some exponent $\alpha$ and every scaling factor $\lambda$,
$\tilde\xi(x):=\lambda^{\alpha-2}\xi(\frac{x_1}{\lambda},\frac{x_2}{\lambda^2})$ has
the same law as $\xi$. We note that this is in line\footnote{ up to a logarithmic loss} 
with assuming that realizations of $\xi$ are H\"older continuous (with respect to our parabolic metric) with exponent $\alpha-2$.
This invariance in the law of the driver $\xi$ yields a covariance of the map
$(a,p)\mapsto(u,q)$ formally defined through \eqref{co10}. Indeed, the above rescaling of
the driver has to be accompanied by a scaling of the input data, namely
$\tilde a(v)=a(\frac{v}{\lambda^\alpha})$ for the nonlinearity
and $\tilde p(x)=\lambda^\alpha p(\frac{x_1}{\lambda},\frac{x_2}{\lambda^2})$
for the imposed large-scale behavior, in order for the solution to react 
according to $\tilde u(x)=\lambda^\alpha u(\frac{x_1}{\lambda},\frac{x_2}{\lambda^2})$
(and similarly for $q$). On the level of the variables defining the input data, 
see \eqref{rm5} and \eqref{co13}, this scaling translates into 
$\z_k[\tilde a,\tilde p]=\lambda^{-\alpha k} \z_k[a,p]$
and $\z_{\bf n}[\tilde a,\tilde p] =\lambda^{\alpha-|{\bf n}|}\z_{\bf n}[a,p]$. In view of our
formal definition of $\mathbf{\Pi}_\beta$ as the partial derivative $\frac{1}{\beta!}\prod_{k\ge 0,\n\not=\0}
(\frac{\partial}{\partial \z_k})^{\beta(k)}
(\frac{\partial}{\partial \z_{\bf n}})^{\beta({\bf n})}u$ this yields the scaling property
\begin{align*}
\tilde{\mathbf{\Pi}}_\beta(x)=\lambda^{|\beta|}\mathbf{\Pi}_\beta(\tfrac{x_1}{\lambda},\tfrac{x_2}{\lambda^2}),
\end{align*}
where 
\begin{equation}\label{hom}
|\beta|:=\alpha\sum_{k\geq 0}k\beta(k)+\sum_{\n\neq 0}(|\n|-\alpha)\beta(\n)+\alpha.
\end{equation}
\begin{remark}\label{remadditive}
	$|\cdot|-\alpha$ is additive, therefore
	\begin{equation}\label{hom4}
	\sum_{i=1}^{k+1}|\beta_{i}|=|\sum_{i=1}^{k+1}\beta_{i}|+k\alpha.
	\end{equation} 
	In particular, 
	\begin{equation}\label{jan04}
		|e_k + \beta_1 +...+\beta_{k+1}| = |\beta_1|+...+|\beta_{k+1}|.
	\end{equation}
\end{remark}
\begin{definition}\label{def:hom} We define the \textbf{set of homogeneities}  as $\dum{A}:=\left\{|\beta|\;|\; \beta\text{ is populated}\right\}$.
\end{definition}
We may rewrite \eqref{hom} in the form
\begin{equation}\label{hom7}
|\beta| = \alpha(1 + [\beta]) + \pol{\beta}
\end{equation}
where $[\beta]$ and $|\beta|_p$ were defined in \eqref{pop12} and \eqref{jan02}, respectively. Note that the second line on the r.~h.~s. of \eqref{pop12} may be rephrased as $[\beta] \geq 0$. 
Thus, for $\alpha>0$, all populated multi-indices satisfy\footnote{The fact that all populated multi-indices have positive homogeneity is, like in footnote \ref{foot01}, related to our model space being ``smaller" than the one in \cite{Ha14,Ha16}: An element $\z_\beta$ for $\beta$ not purely polynomial corresponds to elements in Hairer's model space with homogeneity $|\beta|-2$ and $|\beta|$ (a rooted tree and its planted version, respectively). These two elements are related via Hairer's abstract integration map $\mathcal{I}$; here, they correspond to the r.~h.~s. and the l.~h.~s. of \eqref{fm2}, respectively. In the identification $\T \oplus \tilde{\T} \oplus \R$, the copy $\tilde{\T}$ contains the elements of homogeneity $|\cdot|-2$, whereas $\R$ contains elements of homogeneity $0$. Thus, the full set of homogeneities would be given by $\{\,\kappa,\kappa-2\,|\,\kappa\in\mathsf{A}\}\,\cup\,\{0\}$. Again, we refer to \cite[Subsection 5.3]{LOT21} for details.}
\begin{equation}\label{hom1}
	|\beta|>0
\end{equation}
and, in addition,
\begin{equation}\label{hom2}
	| e_{\n}|=|\n|\text{ for all }\n\neq \0.
\end{equation}
Then \eqref{hom7} and \eqref{hom1} imply $\dum{A}= (\alpha\N_{0}+\N_{0})\setminus\{0\}$. In particular, $\dum{A}$ is bounded from below and locally finite, as it is required in \cite[Definition 3.1]{Ha16}.


\subsection{The centered model}\label{subsec:centeredmodel}
\mbox{}

In order to obtain a local description of the solution manifold near a base point 
$x\in\mathbb{R}^2$, we need what is called a \textbf{centered model} $\Pi_x$. Mimicking
the monomials of a Taylor expansion, the centered model $\Pi_{x\beta}$ 
has the property that it vanishes in $x$ at the order given by the homogeneity\footnote{ cf. \eqref{ap1} below} 
$|\beta|$. As in
\cite[Definition 3.3]{Ha16}, it turns out that the centered model can be recovered from
the pre-model in a completely algebraic way, that is, only by a linear transformation of $\T^*$.

\medskip

We construct the centered model in two steps, first building a two-variable object $\mathbf{\Pi}_x$ with the same features as $\mathbf{\Pi}$ but satisfying an anchoring condition at the base point, and then evaluating this object at the diagonal.
\begin{proposition}\label{prop:centmod} For every $x\in\mathbb{R}^2$ there exist
	\begin{itemize}
	\item $\mathbf{\Pi}_x:\R^2 \times \R^2 \to \mathsf{T}^*$ such that for every $\beta$, $\mathbf{\Pi}_{x\beta}(\cdot,y)$ is a polynomial in $y$ 
with periodic coefficients and
	\item $P_x:\R^2 \to \R[[\z_k,\z_\n]]$ such that for every $\beta$, $P_{x\beta}$ is a polynomial 
	\end{itemize}
	which are related by
	\begin{align}
		&\left(D_{2}-D_{1}^{2}\right)\mathbf{\Pi}_{x}+P_{x}=\mathbf{\Pi}^-_{x},\label{bfp01}\\
		&\mathbf{\Pi}^-_{x}
:=\sum_{k\geq 0}\dum{z}_{k}\mathbf{\Pi}_{x}^{k}D_{1}^{2}\mathbf{\Pi}_{x}+\xi\dum{1},\label{bfp02}\\
		&\text{for all }\beta,\  D^{\n}\mathbf{\Pi}_{x\beta}(x,x)=0 \text{ if }|\n|<|\beta|\label{bfp03}
	\end{align}	
	and, in addition, satisfy
	\begin{align}
		&\text{for all }\n\neq \0,\; \mathbf{\Pi}_{x\beta}(\cdot,y)=	(y-x)^{\n},\\
		&\text{for }\beta\text{ not purely polynomial, }\deg \mathbf{\Pi}_{x\beta} <|\beta|,\label{bfp05}\\
		&\text{for }\beta\text{ populated and not purely polynomial, }\deg P_{x\beta}  <|\beta|-2.\label{bfp06}
	\end{align}
	\begin{proof}
	Let us first be more specific about \eqref{bfp05}: actually, we will show that
	\begin{equation}\label{jan03}
		\text{for }\beta\text{ not purely polynomial, }\mathbf{\Pi}_{x\beta} = \tilde{\mathbf{\Pi}}_{x\beta} + \mbox{polynomial of degree}<|\beta|,
	\end{equation}
	where $\tilde{\mathbf{\Pi}}_{x\beta}$ is a polynomial of degree $\leq|\beta| - \alpha$ with periodic coefficients.
	
	\medskip
	
	The existence of $(\mathbf{\Pi}_{x},P_{x})$ relies on Lemma \ref{lem:wp1} and the constructive algorithm that we proceed to describe. We consider the $\beta$-component of \eqref{bfp01}, i.~e.
	\begin{equation}\label{cco1}
	\left(D_{2}-D_{1}^{2}\right)\mathbf{\Pi}_{x\beta}+P_{x\beta}=\mathbf{\Pi}^-_{x\beta}.
	\end{equation}
	Let us begin with the purely polynomial multi-indices, say $\beta=e_{\m}$, for which \eqref{cco1} reduces to $\left(D_{2}-D_{1}^{2}\right)\mathbf{\Pi}_{xe_{\m}}+P_{xe_{\m}}=0$. For any constants $\{\pi_{xe_{\m}}^{(\n)}\}_{|\n|<|\m|}\subset\R$, the pair
	\begin{align}\label{cco3}
	&(\mathbf{\Pi}_{x e_\m}(\cdot,y),P_{x e_\m}(y))\nonumber \\
	& \quad:= \Big(\mathbf{\Pi}_{e_\m}(\cdot,y) + \sum_{|\n|<|\m|}\pi_{x e_\m}^{(\n)}y^\n, P_{e_\m}(y) - \sum_{|\n|<|\m|}\pi_{x e_\m}^{(\n)} \Big(\frac{\partial}{\partial y_{2}}-\frac{\partial^{2}}{\partial y_{1}^{2}}\Big) y^\n \Big)
	\end{align}	
	clearly solves \eqref{cco1}
and there is a unique  choice of constants such that, 
in addition, $\mathbf{\Pi}_{x e_\m}(\cdot,y)$ $=$ $(y-x)^\m$; 
this choice is given by the binomial formula, i.~e.\footnote{ Here ${\bf m}<{\bf n}$ means (${\bf m}\le{\bf n}$ and ${\bf m}\not={\bf n}$),
		where ${\bf m}\le{\bf n}$ means ($m_1\le n_1$ and $m_2\le n_2$)}
	\begin{align}\label{fm30}
	\pi_{xe_{\bf m}}^{({\bf n})}=\left\{\begin{array}{cl}
	{\tbinom{\bf m}{\bf n}}(-x)^{{\bf m}-{\bf n}}&\mbox{if} \;{\bf n}<{\bf m}\\
	0&\mbox{otherwise}
	\end{array}\right\}\quad\mbox{for}\;{\bf m}\not=\0, {\bf n}.
	\end{align}
		
	\medskip
	
	We now turn to not purely polynomial populated multi-indices, proceeding by induction in the length $\sum_{k\ge0}\beta(k) + \sum_{{\bf n}\not=\0}\beta(\n)$ of the multi-index $\beta$. We start with $\beta=0$ and note that $|0|=\alpha$, cf. \eqref{hom}. By \eqref{fm3}, for any $\{\pi_{x0}^{(\n)}\}_{|\n|<\alpha}\subset\R$, the pair
	\begin{equation}\label{bc01}
	(\mathbf{\Pi}_{x0}(\cdot, y),P_{x0}(y)):=\Big(\mathbf{\Pi}_{0}(\cdot,y)+\sum_{|\n|<\alpha}\pi_{x0}^{(\n)}y^{\n},P_{0}(y)-\sum_{|\n|<\alpha}\pi_{x0}^{(\n)}\Big(\frac{\partial}{\partial y_{2}}-\frac{\partial^{2}}{\partial y_{1}^{2}}\Big)y^{\n}\Big)
	\end{equation}
solves \eqref{cco1}, and there is a unique choice of $\{\pi_{x0}^{(\n)}\}_{|\n|<\alpha}$ 
such that $D^\n\mathbf{\Pi}_{x0}(x,x) = 0$ for all $|\n|<\alpha$.  
Note that by \eqref{jan02}
$\deg\mathbf{\Pi}_{0}=0$, 
so \eqref{jan03} is satisfied. Furthermore, by \eqref{jan01} $P_0 = \fint_{[0,1)^2}\xi$, which vanishes by assumption, and thus \eqref{bfp06} holds.
		
	\medskip	
	
Now suppose that $\beta$ is populated, not purely polynomial and $\beta\neq 0$, 
and assume as induction hypothesis that all $\mathbf{\Pi}_{x \beta'}$ and 
$P_{x\beta'}$, where $\beta'$ is of length strictly smaller than the one of $\beta$, 
have already been constructed and are such that \eqref{bfp06} and \eqref{jan03} are satisfied. 
In particular $\mathbf{\Pi}^{-}_{x\beta}$, cf.~\eqref{bfp02}, is already constructed, 
so that we may appeal to Lemma \ref{lem:wp1} to obtain a solution 
$(\tilde{\mathbf{\Pi}}_{x\beta},\tilde{P}_{x\beta})$ to the problem
	\begin{equation}\label{cco4}
	\left\{\begin{array}{ll}
	(D_{2}-D_{1}^{2})\tilde{\mathbf{\Pi}}_{x\beta}+\tilde{P}_{x\beta}=\mathbf{\Pi}^-_{x\beta},\\
	\fint_{[0,1)^2}\tilde{\mathbf{\Pi}}_{x\beta}=0.
	\end{array}\right.
	\end{equation}
	Let us study some properties of $\mathbf{\Pi}_{x\beta}^-$. We fix $k\geq 0$ and $\{\beta_j\}_{j=1,...,k+1}$ such that $e_k + \beta_1 +... + \beta_{k+1}=\beta$ on the r.~h.~s. of \eqref{bfp02}. By the induction hypothesis \eqref{jan03}, we see that for $k\geq 1$
	\begin{equation}\label{jan05}
		\mathbf{\Pi}_{x\beta_1}\cdots \mathbf{\Pi}_{x\beta_k} = U + \mbox{polynomial of degree}<|\beta_1|+...+|\beta_k|,
	\end{equation}
where $U$ is a polynomial with periodic coefficients with  
$\deg U\leq|\beta_1|+...+|\beta_k| - \alpha$. In addition, for $k\geq 0$, \eqref{jan03} and Leibniz' rule imply the decomposition
	\begin{equation}\label{jan06}
		D_1^2 \mathbf{\Pi}_{x\beta_{k+1}} = U_2 + U_1 + U_0 + \mbox{polynomial of degree}<|\beta_{k+1}|-2,
	\end{equation}
	where
	\begin{itemize}
		\item $U_2$ is a polynomial, the coefficients of which are second spatial derivatives of periodic functions, and $\deg U_2 \leq |\beta_{k+1}|-\alpha$;
		\item $U_1$ is a polynomial, the coefficients of which are first spatial derivatives of periodic functions, and $\deg U_1 \leq |\beta_{k+1}|-\alpha-1$; and
		\item $U_0$ is a polynomial, the coefficients of which are periodic functions, and $\deg U_0 \leq |\beta_{k+1}|-\alpha-2$.
	\end{itemize}
We now focus on the case $k\geq 1$ (the case $k =0$ is much simpler, and does not require multiplying the contributions from \eqref{jan05} and \eqref{jan06}. Note that since\footnote{ The case $\alpha>2$ follows from similar arguments, but we omit it.} $\alpha<2$,
the highest-degree contribution of the product of \eqref{jan05} 
and \eqref{jan06} is obtained by multiplying the polynomial in \eqref{jan05} with $U_2$, 
and thus by \eqref{jan04} we have $\deg \mathbf{\Pi}_{x\beta}^- < |\beta|-\alpha$. 
However, this contribution vanishes under $\fint_{[0,1)^2}$, since its periodic part 
is the derivative of a periodic function. Since $\alpha>1$, the second highest degree 
contribution is obtained multiplying the polynomials in \eqref{jan05} and \eqref{jan06}, 
and by \eqref{jan04} it is of degree $<|\beta|-2$. This implies that 
$\deg \fint_{[0,1)^2}\mathbf{\Pi}_{x\beta}^-<|\beta|-2$. 
Then by Lemma \ref{lem:wp1} all of this implies for the solution of \eqref{cco4}
	\begin{equation}\label{nv3}
	\deg \tilde{\mathbf{\Pi}}_{x\beta}\leq |\beta|-\alpha\;\mbox{and}\; \deg \tilde{P}_{x\beta} < |\beta|-2.
	\end{equation}
	Now for any $\{\pi_{x\beta}^{(\n)}\}_{|\n|<|\beta|}$ the re-centered
	\begin{equation}\label{fm10}
	(\mathbf{\Pi}_{x\beta}(\cdot,y),P_{x\beta}(y)):=\Big(\tilde{\mathbf{\Pi}}_{x\beta}(\cdot,y)+\sum_{|\n|<|\beta|}\pi_{x\beta}^{(\n)}y^{\n},\ \tilde{P}_{x\beta}(y)-\sum_{|\n|<|\beta|}\pi_{x\beta}^{(\n)}\Big(\frac{\partial}{\partial y_{2}}-\frac{\partial^{2}}{\partial y_{1}^{2}}\Big)y^{\n}\Big)
	\end{equation}
	still solves \eqref{cco1}. Obviously, there is a unique choice such that also \eqref{bfp03} holds. By \eqref{nv3} and \eqref{fm10}, \eqref{bfp06} and \eqref{jan03} are satisfied.
	
	\medskip
	
	Finally, suppose that $\beta$ is not populated. As for the pre-model, 
in this case $\mathbf{\Pi}_{x\beta}^-$ is either $0$ or a polynomial. 
Thus, $\mathbf{\Pi}_{x\beta}= 0$ and $P_{x\beta} = \mathbf{\Pi}_{x\beta}^-$ 
solve \eqref{cco1} and satisfy \eqref{bfp03} and \eqref{jan03}.
	\end{proof}
\end{proposition} 
\begin{definition}
	Let $\mathbf{\Pi}_x$ be constructed by Proposition \ref{prop:centmod}, and let $\Pi_x :\R^2 \to \mathsf{T}^*$ be given by
	\begin{equation*}
		\Pi_x(y) := \mathbf{\Pi}_x(y,y).
	\end{equation*}
	We call $\{\Pi_x\}_x$ \textbf{centered model}.
\end{definition}
Our goal now is to give a characterization of $\Pi_x$. For this, we will use the following uniqueness result:

\begin{lemma}\label{lem:uni}
Let $x\in\mathbb{R}^2$, $k\in \N$ and the function $g$ be given. 
There exists at most one pair $(u,p)$ 
of (smooth function,polynomial) with
\begin{equation}\label{jan08}
\left\{\begin{array}{l}
(\partial_2 - \partial_1^2)u + p = g,\\
\partial^\n u(x) = 0\;\mbox{for all }|\n|\leq k,\\
			\limsup_{|y-x|\to \infty} 
\frac{|u(y)|}{|y-x|^{\eta}}<\infty\;\mbox{for some }\eta< k+1,\\
			\deg p \leq k-2.
		\end{array}\right.
	\end{equation}
\end{lemma}

\begin{proof}
	We apply $\partial^\n$ with $|\n|\leq k-2$ to the equation and evaluate at $x$, 
obtaining $\partial^\n p(x) = \partial^\n g(x)$. Thus, by the fourth item, $p$ is determined. 
Consider now two solutions $u_1$, $u_2$ and let $w=u_1-u_2$. Then $(\partial_2-\partial_1^2) w = 0$ and $\limsup_{|y-x|\to\infty}\frac{|w(y)|}{|y-x|^{\eta}}<\infty$. 
By a standard Liouville principle, cf. e.~g. \cite[Exercise 8.4.6]{Kry96}, $w$ is a 
polynomial of degree $\leq$ $\eta$, i.~e. $\leq k$. 
Since all its derivatives at $x$ up to order $k$ are $0$ by the second item, this implies $w=0$.
\end{proof}
\begin{corollary}
$\Pi_x:\R^2 \to \T^*$ is the unique solution of
\begin{equation}\label{fm4}
\left\{ \begin{array}{l}
\left(\partial_{2}-\partial_{1}^{2}\right)\Pi_{x\beta}=\Pi^-_{x\beta}
\text{ for all } \beta\text{ with }[\beta]\geq 0\\
\text{where }\Pi^-_{x\beta}
:=(\sum_{k\geq 0}\dum{z}_{k}\Pi_{x}^{k}\partial_{1}^{2}\Pi_{x}+\xi\dum{1})_\beta,\\
\Pi_{x e_\n}(y)= (y-x)^{\n}\text{ for all }\n\neq\0,\\ 
\partial^{\n}\Pi_{x\beta}(x)=0\text{ provided }|\n|<|\beta|,\\
\limsup_{|y-x|\to\infty}\frac{|\Pi_{x\beta}(y)|}{|y-x|^{|\beta|}}<\infty.
\end{array}\right.
\end{equation}
\end{corollary}
\begin{proof}
	From Lemma \ref{lem:uni} and \eqref{bfp01}--\eqref{bfp06}, it only remains to show that $P_{x\beta}=0$ for all $\beta$ with $[\beta]\geq 0$. By \eqref{bfp06}, we restrict to multi-indices with $|\beta|>2$. We note that, from the proof of Lemma \ref{lem:uni}, we have
	\begin{equation*}
		\partial^\n P_{x\beta}(x) = \partial^\n \Pi_{x\beta}^-(x)\text{ for all }|\n|<|\beta|-2.
	\end{equation*}
	Our goal now is to show that the r.~h.~s. vanishes. By Leibniz' rule, it holds
	\begin{equation*}
		\partial^\n \Pi_{x\beta}^-(x) = \sum_{k\geq 0}\sum_{\n_1 +...+\n_{k+1}=\n}
\hspace{-2ex} \tbinom{\n}{\n_1 ... \n_{k+1}} \hspace{-4ex}
\sum_{e_{k}+\beta_1 + ... + \beta_{k+1}=\beta}\hspace{-3ex}
\partial^{\n_1}\Pi_{x \beta_1}(x)\cdots \partial^{\n_k}\Pi_{x \beta_k}(x) 
\partial^{\n_{k+1} + (2,0)}\Pi_{x \beta_{k+1}}(x).
	\end{equation*}
	Let us fix $k \geq 1$ (the case $k=0$ easily follows by similar arguments). 
The summands for which there exists $i\in\{1,...,k\}$ such that $|\beta_i|>|\n_i|$ 
vanish by the fourth item in \eqref{fm4}. We now focus on those summands for which 
$|\beta_i|\leq |\n_i|$ for all $i\in\{1,...,k\}$. In such a case,
	\begin{equation*}
		|\beta_{k+1}| \underset{\eqref{hom4}}{=} |\beta| - \sum_{i=1}^k |\beta_k| \geq |\beta|- \sum_{i=1}^k |\n_k| = |\beta| - |\n| + |\n_{k+1}|.
	\end{equation*} 
	Then, since $|\n|<|\beta|-2$, this yields
	\begin{equation*}
		|\beta_{k+1}|>|\n_{k+1}|+2 = |\n_{k+1} + (2,0)|.
	\end{equation*}
Hence applying the fourth item in \eqref{fm4}, the last factor vanishes.
\end{proof}
The passage from the pre-model $\mathbf{\Pi}$ to the centered model $\Pi_x$ may be encoded in an algebraic operation involving the coefficients $\{\pi_x^{(\n)}\}_\n \subset \mathsf{T}^*$ chosen in the proof of Proposition \ref{prop:centmod}.
\begin{lemma}\label{lem:algebraic}
Let $x\in\R^{2}$, and let $\{\pi_x^{(\n)}\}_\n\in \mathsf{T}^*$ be constructed
as in the proof of Proposition \ref{prop:centmod}. We introduce the linear subspace $V$
	\begin{equation*}
		V := \{\pi \in \R[[\z_k,\z_\n]]\,|\,\pi_\beta \neq 0 \implies [\beta]\geq -1\}
\supset\mathsf{T}^*.
	\end{equation*}
Assume there exists a linear map\footnote{ At this point, the notation $^*$ is not meaningful; it will become clear at the end of Subsection \ref{subsec:structuregroup}, where we will identify $F_x^*$ with the dual of a map.} $F_x^*:V\to \R[[\z_k,\z_\n]]$ with the properties:  
\begin{enumerate}[label = \textit{(\roman*)}]
\item $F_x^*$ satisfies
\begin{equation}\label{finF}
	\text{ for any }\beta,\;(F_x^* - \textnormal{id})_\beta^\gamma = 0\text{ for all but finitely many }\gamma\text{'s }\text{ with }[\gamma]\geq -1;
\end{equation}	
\item $F_{x}^* \mathsf{T}^* \subset \mathsf{T}^*$;
\item $F_x^*$ acts on the coordinates as follows
\begin{align}
F_{x}^*\dum{z}_{k}&=\sum_{l\geq 0}\left(\begin{matrix}k+l\\l\end{matrix}\right)(\pi_{x}^{(\0)})^{l}\dum{z}_{k+l}\text{ for all }k\geq 0,\label{fm17}\\
F_{x}^*\dum{z}_{\n}&=\dum{z}_{\n}+\pi_{x}^{(\n)}\text{ for all }\n\neq \0; \label{fm19}
\end{align}
\item $F_x^*$ is multiplicative\footnote{ Note that $V$ is not a sub-algebra of $\R[[\z_k,\z_\n]]$.}, i.~e.
\begin{equation}\label{cco5}
F_{x}^*\pi_1\cdots\pi_k=(F_{x}^*\pi_1)\cdots (F_{x}^*\pi_k)\mbox{ for all }\pi_1,...,\pi_k \in V\mbox{ with }\pi_1\cdots\pi_k \in V;
\end{equation}
\item $(F_x^*)_{e_\n}^\gamma \neq 0$ for some ${\bf n}\not=\0$ 
$\implies$ $\gamma$ is purely polynomial.
\end{enumerate}

Then the following relation holds:
\begin{equation}\label{fm20bf}
\mathbf{\Pi}_x = F_x^* \mathbf{\Pi} + \pi_x^{(\0)}
\end{equation}
and as a consequence,
\begin{align}
\Pi_{x}(y)&=F_{x}^*\mathbf{\Pi}(y,y)+\pi_{x}^{(\0)}\label{fm20}.
\end{align} 
\end{lemma}
The argument that such an $F_x^*$ exists
is embedded into the more general study of the structure group in 
Subsection \ref{subsec:structuregroup}, which is a self-contained variant of \cite{LOT21}.
The postulates \textit{(i)} to \textit{(v)} of Lemma \ref{lem:algebraic} slightly differ from the 
properties of \cite[Proposition 5.2]{LOT21}: here $F^*$ is defined on the 
(strictly) larger space $V$. 
The fact that $\T^* \subset V$ is an immediate consequence of \eqref{pop12}; 
the index set of $V$ in addition contains, among others, multi-indices of the form \eqref{pc1}. 
Since in view of Remark \ref{rem:new}, not only $\mathbf{\Pi}$, but also
$P$ and all the summands of  $\mathbf{\Pi}^-$ have values in $V$, 
we may apply $F_x^*$ to both sides of \eqref{fm2}, which is something that under 
\cite[(5.17)]{LOT21} is not possible\footnote{ In \cite{LOTT21}, where the additional polynomial contributions are hidden in the formulation of the model equation, 
it is better not to extend the definition to the larger subspace and estimate 
the multiplicativity defect of $\Pi_x^-$ separately; see e.~g. \cite[(2.63)]{LOTT21} and \cite[Proposition 4.2]{LOTT21} for details.}.

\medskip

\begin{proof}[Proof of Lemma \ref{lem:algebraic}] 
For $\beta$ purely polynomial,
that is, of the form $\beta=e_{\bf m}$ for some ${\bf m}\not={\bf 0}$, 
it follows from \eqref{sm11} and items \textit{(i)} and \textit{(v)} that
$(F_x^*\mathbf{\Pi}+\pi_x^{({\bf 0})})_\beta$ is a polynomial. Hence by 
\eqref{sm11} and \eqref{fm19} 
we have $(F_x^*\mathbf{\Pi}(\cdot,y)+\pi_x^{({\bf 0})})_\beta$ 
$=y^{\bf m}+\sum_{{\bf n}}\pi^{({\bf n})}_{x\beta}y^{\bf n}$. 
It thus follows from \eqref{cco3} that $(F_x^*\mathbf{\Pi}+\pi_x^{({\bf 0})})_\beta$
agrees with $\mathbf{\Pi}_{x\beta}$.

\medskip

For $\beta$ not purely polynomial, we argue
by induction\footnote{the same induction as in the algorithm} in the length of $\beta$. For this we first note that, applying $F_x^*$ to the first two lines in \eqref{fm2},
\begin{align*}
	(D_2 - D_1^2)(F_x^*\mathbf{\Pi} + \pi_x^{(\0)}) + F_x^* P &= F_x^* \big((D_2 - D_1^2)\mathbf{\Pi} + P\big)\\
	&= F_x^* \big(\sum_{k\geq 0} \z_k \mathbf{\Pi}^k D_1^2 \mathbf{\Pi} + \xi\mathsf{1}\big)\\
	&= \sum_{k\geq 0} F_x^*\big(\z_k \mathbf{\Pi}^k D_1^2 \mathbf{\Pi}\big) + \xi\mathsf{1}\\
	&\hspace*{-0.2cm}\underset{\eqref{cco5}}{=} \sum_{k\geq 0} (F_x^* \z_k)(F_x^* \mathbf{\Pi})^k D_1^2 F_x^*\mathbf{\Pi} + \xi\mathsf{1}\\ 
	&\hspace*{-0.2cm}\underset{\eqref{fm17}}{=} \sum_{k\geq 0}\sum_{l\geq 0} \tbinom{k+l}{l}\z_{k+l} (\pi_x^{(\0)})^l (F_x^* \mathbf{\Pi})^k D_1^2 F_x^*\mathbf{\Pi} + \xi\mathsf{1}\\
	&= \sum_{k\geq 0}\z_k (F_x^*\mathbf{\Pi} + \pi_x^{(\0)})^k D_1^2 (F_x^*\mathbf{\Pi} + \pi_x^{(\0)}) + \xi\mathsf{1},
\end{align*}
whereas applying $F_x^*$ to the third line in \eqref{fm2} we obtain
\begin{align*}
	\fint_{[0,1)^2}F_x^*\mathbf{\Pi}(y) = F_x^*\fint_{[0,1)^2}\mathbf{\Pi}(y)\underset{\eqref{fm2}}{=} F_x^* \sum_{\n\neq \0} y^\n \z_\n \underset{\eqref{fm19}}{=} \sum_{\n\neq \0} (\z_\n + \pi_x^{(\n)})y^\n. 
\end{align*}
Note that, as a consequence of \eqref{finF}, for every $\beta$, $(F_x^* P)_\beta = \sum_\gamma (F_x^*)_\beta^\gamma P_\gamma$ is still a polynomial. We start with the base case $\beta=0$. From the above we learn that
\begin{equation*}
	\left\{\begin{array}{l}
		(D_2 - D_1^2)(F_x^*\mathbf{\Pi} + \pi_x^{(\0)})_0  + (F_x^*P)_0= \xi\\
		\fint_{[0,1)^2} (F_x^*\mathbf{\Pi} + \pi_x^{(\0)})_0(y) = \sum_{\n} \pi_{x0}^{(\n)}y^\n.
	\end{array}\right.
\end{equation*}
On the other hand, by \eqref{fm2}, \eqref{cco1} and \eqref{bc01},
\begin{equation*}
	\left\{\begin{array}{l}
		(D_2 - D_1^2)\mathbf{\Pi}_{x0}  + P_{x0}= \xi\\
		\fint_{[0,1)^2} \mathbf{\Pi}_{x0}(y) = \sum_{\n} \pi_{x0}^{(\n)}y^\n.
	\end{array}\right.
\end{equation*}
By the uniqueness statement in Lemma \ref{lem:wp1} with $f = \xi$ and $q(y) = \sum_{\n} \pi_{x0}^{(\n)}y^\n$, $(F_x^*\mathbf{\Pi} + \pi_x^{(\0)})_0=\mathbf{\Pi}_{x0}$, i.~e. \eqref{fm20bf} holds for $\beta = 0$.

\medskip

We now turn to the induction step. Since $\mathbf{\Pi}_{x\beta}^{-}$ depends on $\mathbf{\Pi}_{x\gamma}$ 
only if the length of $\gamma$ is strictly less than the length of $\beta$, 
we have by induction hypothesis \eqref{fm20bf} that the r.~h.~s.~of \eqref{cco4} agrees with 
$(D_2-D_1^2)(F_x^*\mathbf{\Pi}+\pi_x^{({\bf 0})})_\beta + (F_x^*P)_\beta$ above. 
Hence in view of \eqref{fm10}, both $\mathbf{\Pi}_{x\beta}$ and $(F_x^*\mathbf{\Pi}+\pi_x^{({\bf 0})})_\beta$
solve \eqref{staronp9} with $f = \mathbf{\Pi}_{x\beta}^-$ and $q(y)=\sum_{{\bf n}}\pi^{({\bf n})}_{x\beta}y^{\bf n}$; we conclude once more by Lemma \ref{lem:wp1}.
\end{proof}

%
\subsection{The structure group}\label{subsec:structuregroup}
\mbox{}

The existence of endomorphisms of $\T^*$ satisfying \eqref{cco5}, \eqref{fm17} and \eqref{fm19} was established in full generality in \cite{LOT21}, using techniques of Hopf algebras and representation of groups. In this subsection we give a self-contained construction, following the approach for the restricted 
setting of Section \ref{sec:restrictedmodel}.

\medskip

We now assume that we are given general 
\begin{equation}\label{sg14}
\left\{\pi^{(\n)}\right\}_{\n}\subset\dum{T}^*,
\end{equation}
which in line with \eqref{fm10} and \eqref{fm30} satisfy 
\begin{align}
\pi_{\beta}^{(\n)}&\neq 0\implies |\n|<|\beta|,\label{sg9}\\
%
%
\pi_{e_{\bf n}}^{({\bf m})}&=\left\{\begin{array}{cl}
{\tbinom{\bf n}{\bf m}}h^{{\bf n}-{\bf m}}&\mbox{if}\;{\bf m}<{\bf n}\\
0&\mbox{otherwise}
\end{array}\right\}\quad\mbox{for}\; {\bf n}\not=\0, \m\label{sg20}
\end{align}
for some shift vector $h\in\mathbb{R}^2$.
Note that in view of \eqref{hom2}, \eqref{sg9} and \eqref{sg20} are consistent.
Under these assumptions we will construct an endomorphism $\Gamma^*$ satisfying
\begin{align}
&\Gamma^* \dum{z}_{k}=\sum_{l\geq 0}\left(\begin{matrix}k+l\\k\end{matrix}\right)(\pi^{(\0)})^{l}\dum{z}_{k+l}\text{ for all }k\geq 0,\label{sg1}\\
&\Gamma^* \dum{z}_{\n}=\dum{z}_{\n}+\pi^{(\n)}\text{ for all }\n\neq \0.\label{sg2}
\end{align}
Conditions \eqref{sg1} and \eqref{sg2} mimic \eqref{fm17} and \eqref{fm19}, respectively, 
and characterize $\Gamma^*$ on $\R[\dum{z}_{k},\dum{z}_{\n}]$ 
(but not on $\R[[\dum{z}_{k},\dum{z}_{\n}]]$ or $\dum{T}^*$).

\medskip

The next lemma establishes that this is achieved by the exponential-type formula
%
\begin{equation}\label{sg3}
\Gamma^*=\sum_{l\geq 0}\frac{1}{l!}\sum_{\n_{1},...,\n_{l}}\pi^{(\n_{1})}\cdots\pi^{(\n_{l})}D^{(\n_{1})}\cdots D^{(\n_{l})},
\end{equation}
where we recall that $D^{(\0)}$ is formally given by \eqref{der1} and rigorously by \eqref{der2}, 
and where $D^{(\n)}$ is formally given by 
\begin{equation}\label{*p.18}
D^{(\n)}=\partial_{\dum{z}_{\n}}\text{ for all }\n\neq\0,
\end{equation}
 which on the level of coefficients means
\begin{equation}\label{der3}
(D^{(\n)})_{\beta}^{\gamma}=\left\{\begin{array}{ll}
\gamma(\n) & \text{ if }\beta+e_{\n}=\gamma,\\
0 & \text{ otherwise}.
\end{array}\right.
\end{equation}
Note that, as in case of $D^{(\0)}$, $D^{(\n)}$ for $\n\neq \0$ is a derivation, cf. \cite[Example 1.8]{Abe80}, and for a fixed $\beta$ there are finitely many $\gamma$'s (in fact only one) such that $(D^{(\n)})_{\beta}^{\gamma}\neq 0$, and therefore the derivation $D^{(\n)}$ is well-defined as a linear endomorphism of $\R[[\dum{z}_{k},\dum{z}_{\n}]]$. We point out here that the derivations $\{D^{(\n)}\}_{\n}$ commute, i.~e.
\begin{equation}\label{dercom}
D^{(\n)}D^{(\n')}=D^{(\n')}D^{(\n)}\text{ for all }\n,\n'.
\end{equation} 
Indeed, \eqref{dercom} is obvious for $\n,\n'\neq \0$ from \eqref{*p.18}. For $\n'=\0$ and $\n\neq \0$, in view of \eqref{der1}, \eqref{dercom} follows because the derivation $\partial_{\dum{z}_{\n}}$ and the multiplication with $\dum{z}_{k+1}$ commute.

\medskip

\begin{lemma}\label{lemts2} Given 
$\{\pi^{(\n)}\}_\n\subset\mathsf{T}^*$ satisfying \eqref{sg9} and \eqref{sg20},
	\begin{enumerate}[label = \textit{(\roman*)}]
		\item $\Gamma^*|_{V} : V \to \R[[\z_k,\z_\n]]$ is well-defined through \eqref{sg3} and, in addition, it satisfies the finiteness property
		\begin{align}
			\text{for any }\beta,\ (\Gamma^*-{\rm id})_{\beta}^{\gamma}=0\text{ for all but finitely many }\gamma{'s}\text{ with }[\gamma]\geq -1;\label{stat1bis}
		\end{align}
		\item $\Gamma^*|_{\mathsf{T}^*} \in \textnormal{End}(\mathsf{T}^*)$;
		\item $\Gamma^*$ satisfies \eqref{sg1} and \eqref{sg2};
		\item $\Gamma^*$ is strictly triangular with respect to the homogeneity, i.~e.
		\begin{equation}\label{sg5}
			(\Gamma^*-\textnormal{id})_{\beta}^{\gamma}\neq0\,\implies\,|\gamma|<|\beta|;
		\end{equation}
	\item $\Gamma^*$ is multiplicative, i.~e.
	\begin{equation}\label{multgamma}
	\Gamma^*\pi_1\cdots\pi_k=(\Gamma^*\pi_1)\cdots (\Gamma^*\pi_k)\mbox{ for all }\pi_1,...,\pi_k \in V\mbox{ with }\pi_1\cdots\pi_k \in V;
	\end{equation}
	\item for all ${\bf n}\not=\0$
	\begin{align}\label{cw01}
		(\Gamma^*)_{e_{\bf n}}^{\gamma}
		=\left\{\begin{array}{cl}
			{\tbinom{\bf n}{\bf m}}h^{{\bf n}-{\bf m}}&\mbox{if}\;\gamma=e_{\bf m}\;
			\mbox{for some}\;\0\not={\bf m}\le{\bf n},\\
			0&\mbox{otherwise.}
		\end{array}\right.
	\end{align}
	\end{enumerate}
\end{lemma}
We now give an interpretation of \eqref{cw01}.
Let $\Ppol$ denote the projection onto $\mathsf{\bar T}^*$, i.e. 
\begin{equation*}
(\Ppol\pi)_{\beta}=\left\{\begin{array}{ll}
\pi_{\beta} & \text{ if }\beta=e_\n \text{ for some }\n\neq \0,\\
0 & \text{ otherwise.}
\end{array}\right.
\end{equation*}
By the canonical
identification
$\mathsf{\bar T}\cong \mathbb{R}[y_1,y_2]/\mathbb{R}$ given through \eqref{cw03},
this definition can be rephrased as
\begin{align}\label{cw06}
\Ppol\pi [p]
=\sum_{{\bf n}\not=0}\pi_{e_{\bf n}}\frac{1}{{\bf n}!}
\frac{\partial^{\bf n}p}{\partial y^{\bf n}}(0).
\end{align}
We now argue that \eqref{cw01} takes the basis-free form of
\begin{align}\label{cw04}
\Ppol \Gamma^* \pi [p]=\Ppol \pi[p(\cdot + h)]\quad\mbox{for all}\;
p\in\mathbb{R}[y_1,y_2]\;\mbox{and}\;\pi\in\mathsf{T}^*.
\end{align}
Indeed, with the abbreviation $\tilde p:=p(\cdot+h)$ we have by Taylor's
\begin{align*}
\Ppol\Gamma^*\pi [p]
&
\underset{\eqref{cw06}}{=}\sum_{{\bf n}\not=0}\sum_{\gamma}(\Gamma^*)_{e_{\bf n}}^\gamma\pi_\gamma
\frac{1}{{\bf n}!}\frac{\partial^{\bf n}p}{\partial y^{\bf n}}(0)\\
&\underset{\eqref{cw01}}{=}
\sum_{{\bf n}\not=0}
\sum_{\0\not={\bf m}\le{\bf n}}{\tbinom{\bf n}{\bf m}}h^{\n-\m}\pi_{e_{\bf m}}
\frac{1}{{\bf n}!}\frac{\partial^{\bf n}p}{\partial y^{\bf n}}(0)\\
&=\sum_{{\bf m}\not=0}\pi_{e_{\bf m}}
\frac{1}{{\bf m}!}\frac{\partial^{\bf m}\tilde p}{\partial y^{\bf m}}(0)
\underset{\eqref{cw06}}{=}\Ppol\pi[\tilde p].
\end{align*}
Identity \eqref{cw04} means that the action of 
$\Gamma^*$ on the subspace $\mathsf{\bar T}^*$ is the pull back
of the action of shift on $\mathbb{R}[y_1,y_2]$.
This amounts to the postulate \cite[Assumption 3.20]{Ha16} on the polynomial sector.

\medskip

In the proof of Lemma \ref{lemts2} we will make use of the following property of 
the matrix representation of our derivations.
\begin{lemma}\label{lemder}
	Fix $\beta$, $\gamma$, $l\geq 1$ and $\n_{1},...,\n_{l}$. Then\footnote{ Here we use the convention $|\0| = 0$.}
	\begin{equation}\label{sg8}
	\left(D^{(\n_{1})}\cdots D^{(\n_{l})}\right)_{\beta}^{\gamma}\neq 0\implies[\beta] = [\gamma] + l \;\;\mbox{and}\;\; \pol{\beta} + \sum_{k=1}^l |\n_k| = \pol{\gamma}.
	\end{equation}
\end{lemma}
\begin{proof}[Proof of Lemma \ref{lemder}]
	We proceed by induction in $l\geq 1$. We start with the base case $l=1$. 
If $\n=\0$, then $(D^{(\0)})_{\beta}^{\gamma}\neq 0$ implies $\beta+e_{k}=\gamma+e_{k+1}$ 
for some $k\geq 0$, cf. \eqref{der2}. Therefore $[\beta]=[\gamma]+1$ and 
$\pol{\beta} = \pol{\gamma}$, as desired. 
Now if $\n\neq\0$ then $(D^{(\n)})_{\beta}^{\gamma}\neq 0$ implies $\beta+e_{\n}=\gamma$, 
cf. \eqref{der3}, and consequently $[\beta]=[\gamma]+1$ and $\pol{\beta} + |\n| = \pol{\gamma}$,
as desired.
	
	\medskip
	
	Now assume that \eqref{sg8} is satisfied for some $l$. We write
	\begin{equation*}
	\left(D^{(\n_{1})}\cdots D^{(\n_{l+1})}\right)_{\beta}^{\gamma}=\sum_{\tilde{\gamma}}\left(D^{(\n_{1})}\cdots D^{(\n_{l})}\right)_{\beta}^{\tilde{\gamma}}\left(D^{(\n_{l+1})}\right)_{\tilde{\gamma}}^{\gamma}.
	\end{equation*}
	From the base case we learn that there has to be a $\tilde{\gamma}$ such that $[\tilde{\gamma}]=[\gamma] +1$ and $\pol{\tilde{\gamma}} + |\n_{l+1}| = \pol{\gamma}$. On the other hand, by the induction hypothesis, $[\beta] = [\tilde{\gamma}] + k$ and $\pol{\beta} + \sum_{k= 1}^l |\n_{k}|=\pol{\tilde{\gamma}}$. The combination of these statements yields \eqref{sg8} for $l+1$.
\end{proof}	

\begin{proof}[Proof of Lemma \ref{lemts2}]
		Throughout the proof we will make use of the matrix representation of \eqref{sg3}: for multi-indices $\beta$, $\gamma$,
		\begin{equation}\label{sg7}
		(\Gamma^* - \textnormal{id})_{\beta}^{\gamma}=\sum_{l\geq 1}\frac{1}{l!}\sum_{\n_{1},...,\n_{l}}\sum_{\beta_{1}+...+\beta_{l+1}=\beta}\pi_{\beta_{1}}^{(\n_{1})}\cdots\pi_{\beta_{l}}^{(\n_{l})}\left(D^{(\n_{1})}\cdots D^{(\n_{l})}\right)_{\beta_{l+1}}^{\gamma}.
		\end{equation}
		In order to prove that $\Gamma^*|_V$ is well-defined, it is enough to verify \eqref{stat1bis} together with
		\begin{align}
			\text{for any }\beta\text{ and any }\gamma\text{ with }[\gamma]\geq -1,\text{ the sum in }\eqref{sg7}\text{ is finite.}\label{stat2bis}
		\end{align}
		We fix multi-indices $\beta$ and $\gamma$ with $[\gamma]\geq -1$ and look at 
the component-wise expression \eqref{sg7}. We apply Lemma \ref{lemder} to 
$(D^{(\n_{1})}\cdots D^{(\n_{l})})_{\beta_{l+1}}^{\gamma}$; from \eqref{sg8} we learn 
that $|\gamma|_p \geq \sum_{k=1}^l |\n_{k}|$ which by \eqref{hom7} and $[\gamma]\geq -1$ implies
		\begin{equation}\label{sg13}
			|\gamma|\geq \sum_{k=1}^l |\n_{k}|.
		\end{equation}  

Now due to $\beta_{1}+...+\beta_{l+1}=\beta$ and with help of \eqref{hom4}
we may rewrite \eqref{sg8} as
		\begin{equation}\label{sg16}
			|\beta|+\sum_{k=1}^l |\n_{k}|=|\gamma|+\sum_{k=1}^l |\beta_{k}|.
		\end{equation}
The combination of \eqref{sg13} and \eqref{sg16} yields $|\beta|\geq\sum_{k=1}^l |\beta_{k}|$.
Since by \eqref{sg14}, we may assume that all $\beta_{k}$ are populated,  
and therefore have homogeneity uniformly bounded away from zero, cf. \eqref{hom1}, 
$l$ is bounded by above in terms of $|\beta|$,
and thus the sum \eqref{sg7} is effectively finite in $l$. 
Now for a fixed $l$ there are finitely many ways of decomposing
$\beta=\beta_{1}+...+\beta_{l+1}$, and for each $\beta_{k}$ there are finitely 
many $\n_{k}$'s such that $|\n_{k}|<|\beta_{k}|$, see \eqref{polyn}. 
In view of \eqref{sg9}, this implies that the whole sum in \eqref{sg7} is finite 
for every $\beta$, which concludes the proof of \eqref{stat2bis}. 
Finally, since for every $\beta_{l+1}$ there are finitely many $\gamma$'s 
such that $\left(D^{(\n_{1})}\cdots D^{(\n_{l})}\right)_{\beta_{l+1}}^{\gamma}\neq 0$, 
we conclude that for every fixed $\beta$ there are finitely many $\gamma$'s 
such that $(\Gamma^*-{\rm id})_{\beta}^{\gamma}\neq 0$, which establishes \eqref{stat1bis} and finishes the proof of item \textit{(i)}.	
		
		\medskip
		
For $\Gamma^*|_{\mathsf{T}^*}\in\textnormal{End}(\dum{T}^*)$ we have to check that
\begin{align}
\text{if $\gamma$ is populated and }(\Gamma^*-{\rm id})_{\beta}^{\gamma}\neq 0\text{ , then
$\beta$ is populated.}\label{sg18}
\end{align}
For later use, we will show this property for each of the summands in \eqref{sg3}, i.~e. for all $l\geq 1$ and $\n_1,...,\n_l$,
\begin{equation}\label{cco7}
	\text{if $\gamma$ is populated and }(\pi^{(\n_{1})}\cdots\pi^{(\n_{l})}D^{(\n_{1})}\cdots D^{(\n_{l})})_\beta^{\gamma}\neq 0\text{ , then
		$\beta$ is populated.}
\end{equation}
For $\gamma$ purely polynomial, say $\gamma = e_\n$, $(\pi^{(\n_{1})}\cdots\pi^{(\n_{l})}D^{(\n_{1})}\cdots D^{(\n_{l})})_\beta^{e_\n}\neq 0$ implies $l=1$ and $\n_1 = \n$, and in this case $(\pi^{(\n)}D^{(\n)})_\beta^{e_\n} = \pi_\beta^{(\n)}$, so $\beta$ is populated, cf. \eqref{sg14}.
If $\gamma$ is not purely polynomial and populated, we expand \eqref{cco7} in form of
\begin{equation*}
	\sum_{\beta_1 + ...+\beta_{l+1} = \beta} \pi_{\beta_1}^{(\n_1)}\cdots \pi_{\beta_l}^{(\n_l)}(D^{(\n_1)}\cdots D^{(\n_l)})_{\beta_{l+1}}^\gamma.
\end{equation*}
We combine $\beta = \beta_1 +... +\beta_{l+1}$ with the first item in \eqref{sg8} to obtain
$[\beta]= [\gamma] + \sum_{k=1}^l [\beta_{k}] + l$. Since the 
$\beta_1,\cdots,\beta_l$ may be assumed to be populated because of \eqref{sg14}, 
and we treat the case of a populated $\gamma$ that is not purely polynomial, we have
$[\beta_1],\cdots,[\beta_l]\ge -1$ and $[\gamma]\ge 0$ by \eqref{pop12}, and thus
$[\beta]\ge 0$. Hence $\beta$ is populated, which establishes \eqref{cco7}, and 
moreover not purely polynomial. For later purpose, we 
retain the logically equivalent
\begin{align}\label{ao77}
\mbox{if $\gamma$ is populated and }
(\Gamma^*)_{e_{\bf n}}^\gamma\not=0\;\mbox{for some}\;{\bf n}\not=\0,
\mbox{then }\gamma=e_{\bf m}\;\mbox{for some}\;{\bf m}\not=\0.
\end{align}

		\medskip
		
Note that \eqref{sg1} and \eqref{sg2} follow immediately 
from \eqref{sg3} by \eqref{dercom}, and by the characterizing properties 
\eqref{*p.18} and \eqref{der1}, 
which by induction implies $(D^{({\bf 0})})^l\mathsf{z}_k$ $=\tbinom{k+l}{l}\mathsf{z}_{k+l}$. 
Also the argument for \eqref{sg5} is easy: 
In view of \eqref{sg7}, the statement $(\Gamma^*-\mbox{id})_{\beta}^{\gamma}\neq 0$ implies that for some $l\geq 1$ and $\n_{1},...,\n_{l}$ 
		\begin{equation*}
		|\gamma|+l\alpha\underset{\eqref{sg8}}{=}|\beta_{l+1}|+\sum_{l'=1}^l |\n_{l'}|\underset{\eqref{sg9}}{<}|\beta_{l+1}|+\sum_{l'=1}^l |\beta_{l'}|\underset{\eqref{hom4}}{=}|\beta|+l\alpha.
		\end{equation*}
		
		\medskip
		
We turn to the proof of \eqref{multgamma}. To this
purpose, it is convenient to capitalize on \eqref{dercom} in order to re-sum \eqref{sg3}:
\begin{align}\label{komb}
\Gamma^*=\sum_{J}\Gamma^*_J\mbox{ where }
\Gamma^*_J:=\frac{1}{J!}\big(\prod_{{\bf n}}(\pi^{({\bf n})})^{J(\n)}\big)
\big(\prod_{\n}(D^{({\bf n})})^{J(\n)}\big),
\end{align}
where $J$ runs over all multi-indices of $\mathbb{N}_0^2$.
The ratio between the combinatorial $l!$ where $l:=\sum_{{\bf n}}J(\n)$,
and the combinatorial $J!:=\prod_{\n}(J(\n)!)$ is the number of ways of ordering 
the $l$ not necessarily distinct objects identified by $J$.
We now appeal to the iterated\footnote{the combinatorial factors can be checked on a standard
derivative by considering monomials} Leibniz' rule
on the level of the derivation $D^{(\n)}$ in form of
\begin{align*}
\frac{1}{J(\n)!}(D^{(\n)})^{J(\n)}\pi_1\cdots\pi_k
=\sum_{j_1+...+j_k = J(\n)}\frac{1}{j_1!}((D^{(\n)})^{j_1}\pi_1)\cdots
\frac{1}{j_k!}((D^{(\n)})^{j_k}\pi_k),
\end{align*}
which on the level of \eqref{komb} combines to
\begin{align*}
\Gamma_J^*\pi_1\cdots\pi_k=\sum_{J_1 +...+J_k = J}(\Gamma_{J_1}^*\pi_1)\cdots (\Gamma_{J_k}^*\pi_k).
\end{align*}
Now \eqref{multgamma} follows immediately from \eqref{komb},
where we use the previously established property that all the sums are effectively finite.

\medskip
		
		We finally turn to \eqref{cw01}. Due to \eqref{ao77}, it only remains to show for ${\bf m},{\bf n}\not=0$
		\begin{align*}
		(\Gamma^*)_{e_{\bf n}}^{e_{\bf m}}
		=\left\{\begin{array}{cl}
		{\tbinom{\bf n}{\bf m}}h^{{\bf n}-{\bf m}}&\mbox{if}\;{\bf m}\le{\bf n},\\
		0&\mbox{otherwise,}
		\end{array}\right.
		\end{align*}
		which follows from \eqref{sg2} and \eqref{sg20}.
\end{proof}

\medskip

We claim that these $\Gamma^*$'s form a group of endomorphisms of $\dum{T}^*$. As in Subsection \ref{subsec:shifts}, the triangular structure \eqref{sg5} implies the
invertibility of $\Gamma^*$. It only remains to establish the composition rule.
\begin{lemma}\label{lemcomp} Let $\Gamma^*$, $\Gamma'^*\in\textnormal{End}(\dum{T}^*)$ be given 
via \eqref{sg3} by $\{\pi^{(\n)}\in\dum{T}^*\}_{\n}$, $\{\pi'^{(\n)}\in\dum{T}^*\}_{\n}$, 
respectively. Then the composition $\Gamma^*\Gamma'^*\in\textnormal{End}(\dum{T}^*)$ is given 
by $\{\pi^{(\n)}+\Gamma^*\pi'^{(\n)}\in\dum{T}^*\}_{\n}$. Moreover, 
$\{\pi^{(\bf n)}+\Gamma^*\pi'^{(\bf n)}\}_{\bf n\not=0}$ satisfies
\eqref{sg9} and \eqref{sg20}, the latter with the shift $h+h'$.
\begin{proof} Fixing $\pi\in\dum{T}^*$, we have
\begin{align*}
\Gamma^*\Gamma'^*\pi&\underset{\eqref{sg3}}{=}\Gamma^*\sum_{l'\geq 0}\frac{1}{l'!}\sum_{\n_{1}',...,\n_{l'}'}\pi'^{(\n_{1}')}\cdots\pi'^{(\n_{l'}')}D^{(\n_{1}')}\cdots D^{(\n_{l'}')}\pi\\
&\hspace*{-0.43cm}\underset{\eqref{multgamma},\eqref{cco7}}{=}\sum_{l'\geq 0}\frac{1}{l'!}\sum_{\n_{1}',...,\n_{l'}'}(\Gamma^*\pi'^{(\n_{1}')})\cdots(\Gamma^*\pi'^{(\n_{l'}')})\Gamma^*(D^{(\n_{1}')}\cdots D^{(\n_{l'}')}\pi)\\
&\underset{\eqref{sg3}}{=}\hspace*{-3pt}\sum_{l'\geq 0}\sum_{l\geq 0}\frac{1}{l'!l!}\hspace*{-5pt}\sum_{\substack{\n_{1}',...,\n_{l'}'\\\n_{1},...,\n_{l}}}\hspace*{-4pt}(\Gamma^*\pi'^{(\n_{1}')})\cdots(\Gamma^*\pi'^{(\n_{l'}')})\pi^{(\n_{1})}\cdots\pi^{(\n_{l})}D^{(\n_{1})}\cdots D^{(\n_{l})}D^{(\n_{1}')}\cdots D^{(\n_{l'}')}\hspace*{-1pt}\pi
\end{align*}
and after re-summation, and since by Lemma \ref{lemts2} all sums are finite,  this equals
\begin{align*}
&\sum_{l\geq 0}\frac{1}{l!}\sum_{\n_{1},...,\n_{l}}\bigg(\sum_{l'=0}^{l}\left(\begin{matrix}l\\l'\end{matrix}\right)(\Gamma^*\pi'^{(\n_{1})})\cdots(\Gamma^*\pi'^{(\n_{l'})})\pi^{(\n_{l'+1})}\cdots\pi^{(\n_{l})}\bigg)D^{(\n_{1})}\cdots D^{(\n_{l})}\pi\\
=&\sum_{l\geq 0}\frac{1}{l!}\sum_{\n_{1},...,\n_{l}}(\Gamma^*\pi'^{(\n_{1})}+\pi^{(\n_{1})})\cdots(\Gamma^*\pi'^{(\n_{l})}+\pi^{(\n_{l})})D^{(\n_{1})}\cdots D^{(\n_{l})}\pi,
\end{align*}
which by \eqref{sg3} completes the proof of the composition rule. 

\medskip

The preservation of \eqref{sg9} is immediate from the triangular structure \eqref{sg5}. 
It only remains to show that \eqref{sg20} is preserved. Indeed, 
for $\m<\n$ we have
\begin{align}
\pi_{e_{\bf n}}^{({\bf m})}+\sum_{\gamma}(\Gamma^*)_{e_{\bf n}}^{\gamma}{\pi'}_{\gamma}^{({\bf m})}
	&\underset{\eqref{cw01}}{=}
	\pi_{e_{\bf n}}^{({\bf m})}+
	\sum_{\0\neq\n'\leq \n}\binom{\n}{\n'}h^{{\bf n}-{\bf n}'}
	{\pi'}_{e_{\bf n'}}^{({\bf m})}\label{jul01}\\
&\underset{\eqref{sg20}}{=}
\binom{\n}{\m}h^{\n-\m} + \sum_{\m<\n'\leq \n}\binom{\n}{\n'} 
h^{\n-\n'}\binom{\n'}{\m}{h'}^{\n'-\m}\nonumber\\
&\hspace*{0.2cm}= \sum_{\m\leq\n'\leq\n}\binom{\n}{\n'}\binom{\n'}{\m}h^{\n-\n'}{h'}^{\n'-\m},\nonumber
\end{align}
so that we may conclude by the identity
$\binom{\n}{\n'}\binom{\n'}{\m}$ $=\binom{\n}{\m}\binom{\n-\m}{\n'-\m}$ and
the binomial formula. If $\m<\n$ fails, also $\m<\n'$ in the sum
in line \eqref{jul01} fails, so that the the entire r.~h.~s.~in this line vanishes
according to \eqref{sg20}, as desired.
Note that the preservation of \eqref{sg20}
is even more obvious on the level of the characterization \eqref{cw04}:
\begin{align*}
\mathsf{P}\Gamma^*\Gamma'^*\pi[p]=\mathsf{P}\Gamma'^*\pi[p(\cdot + h)]
=\mathsf{P}\pi[p(\cdot + h+ h')].
\end{align*}
\end{proof}
\end{lemma}

\begin{definition}\label{defsg} We denote by $\dum{G}^*\subset\textnormal{End}(\dum{T}^*)$ 
the group of endomorphisms $\Gamma^*$ given by \eqref{sg3} 
with $\left\{\pi^{(\n)}\in \dum{T}^*\right\}_{\n}$ satisfying \eqref{sg9} and \eqref{sg20}.
\end{definition}
As for $\mathsf{T}^*$, the notation $\mathsf{G}^*$ is chosen because it is
the pointwise dual of a group $\mathsf{G}\subset{\rm End}(\mathsf{T})$
that may be interpreted as the \textbf{structure group} in regularity structures. 
In particular, it follows from the last part
of the statement of Lemma \ref{lemcomp} that $\mathsf{G}$ acts on $\mathsf{\bar{T}}\subset \mathsf{T}$
in a way that is isomorphic to the action of $\mathbb{R}^2$ on $\mathbb{R}[y_1,y_2]/\mathbb{R}$.

\medskip

Summing up, the objects $(\mathsf{A}, \mathsf{T}^*,\mathsf{G}^*)$, cf. Definitions \ref{def:hom}, \ref{def:modspace} and \ref{defsg}, are dual to a regularity structure 
$(\mathsf{A},\mathsf{T},\mathsf{G})$. In our setting, however, we will make no use of 
this ``primal" side, and only the dual side will be relevant.

\medskip

Let us now return to $F_{x}^*$, which is the endomorphism that centers the pre-model at $x\in\R^{2}$, cf. \eqref{fm20}. By the construction provided in Subsection \ref{subsec:centeredmodel}, $F_{x}^*\in\dum{G}^*$, so in particular it is invertible. Thus, we may describe the change of the base point in our centered model via the structure group, and that is precisely the task of the \textit{re-expansion maps} $\Gamma_{xy}^*$. Let $x,y\in\R^{2}$. Using \eqref{fm20} and the invertibility of $F_{x}^*$, we obtain
\begin{equation}\label{**p.20}
\Pi_{y}(z)=F_{y}^*\mathbf{\Pi}(z,z)+\pi_{y}^{(\0)}=F_{y}^*F_{x}^{*-1}\Pi_{x}(z)-F_{y}^*F_{x}^{^*-1}\pi_{x}^{(\0)}+\pi_{y}^{(\0)}.
\end{equation}
Defining 
\begin{equation}\label{*p.20}
\Gamma_{yx}^*:=F_{y}^*F_{x}^{*-1},
\end{equation} 
which, according to Lemma \ref{lemcomp}, is generated by
\begin{equation*}
\{\pi_{yx}^{(\n)}\}_{\n}:=\{\pi_{y}^{(\n)}-F_{y}^*F_{x}^{*-1}\pi_{x}^{(\n)}\}_{\n},
\end{equation*}
we see that \eqref{**p.20} takes the form
\begin{equation}\label{sg6}
\Pi_{y}=\Gamma_{yx}^*\Pi_{x}+\pi_{yx}^{(\0)}.
\end{equation}
For later reference we note that, by the fourth item in \eqref{fm4}, this implies in particular 
\begin{equation}\label{sg15}
\pi_{yx}^{(\0)}=\Pi_{y}(x).
\end{equation}
Definition \eqref{*p.20}, moreover, implies transitivity in form of
\begin{equation}\label{sg12}
\Gamma_{zx}^*=\Gamma_{zy}^*\Gamma_{yx}^*.
\end{equation}
The two relations \eqref{sg6} and \eqref{sg12} are the algebraic identities postulated in \cite[Definition 3.3]{Ha16} for the model.

\begin{remark}\label{remmodel}
	Note that \eqref{sg6} differs from the analogue identity in \cite[Definition 3.3]{Ha16} in the appearance of the constant $\pi_{yx}^{(\0)}$. The presence of $\pi_{yx}^{(\0)}$ is due to our restricted definition of the model space. Whereas in \cite[Section 4.2]{Ha16} $\dum{T}$ 
includes the neutral element $\dum{1}$ of homogeneity $0$, our model space does not 
include these elements (which we could identify with the ``missing" variable $\z_{\n=\0}$). See \cite[Subsection 5.3]{LOT21} for more details.\end{remark}

\medskip

\section{Estimates of the model}\label{sec:analyticalproperties}
In this section we establish the bounds on the model
axiomatized in \cite[Definition 3.3]{Ha16}.
The parabolic scaling of $\partial_2 - \partial_1^2$
imposes the Carnot-Carathéodory distance \eqref{dist01}. It is convenient to work with a family 
$\{(\cdot)_t\}_{t>0}$ of convolution operators compatible with the parabolic scaling
and in addition satisfying the semi-group 
property 
\begin{align}\label{ud04}
((\cdot)_t)_T=(\cdot)_{T+t}.
\end{align}
As in \cite[Section 2]{OW19}, we take $(\cdot)_t=\psi_t$,
where $\psi_t$ is the kernel of the semi-group generated by the space-time elliptic operator 
$\partial_1^4 - \partial_2^2$.
This operator is the squared modulus 
$(\partial_2-\partial_1^2)^*(\partial_2-\partial_1^2)$ of the original operator
and thus respects the parabolic scaling, next to being non-negative and symmetric.
The fact that it contains $\partial_2-\partial_1^2$ as a factor will allow us to represent
the kernel of $(\partial_2-\partial_1^2)^{-1}$ in terms of $(\cdot)_t$, 
which is convenient in the integration step, see Lemma \ref{lem:int}.
Besides these structural properties, we need
the following moment bounds\footnote{which are an easy consequence
of the fact of scaling and that $\psi_{t=1}$ has Fourier transform 
$\exp(-k_1^4-k_2^2)$ and thus is a Schwartz function}
\begin{equation}\label{ker02}
	\int dy |\partial^\n \psi_t (x-y)||y-x|^\eta \lesssim (\sqrt[4]{t})^{\eta - |\n|}.
\end{equation}
Convolutions can be used to characterize negative (parabolic) H\"older spaces,
(cf. \cite[Lemma A.1]{OW19}). We thus formulate our main regularity assumption as follows
\begin{equation}\label{normxi}
N_0:= \sup_t (\sqrt[4]{t})^{2-\alpha}\sup_{y}|\xi_t(y)|<\infty.
\end{equation}

\begin{theorem}\label{propanalyticbounds}
Let $\alpha\in(1,2)$ with $\alpha\notin \Q$. 
Let $\xi$ be smooth, periodic and satisfy \eqref{normxi}. 
Then for all points $x$, $y\in\mathbb{R}^{2}$, the following estimates hold:
	\begin{align}
		|\Pi_{x\beta}(y)|&\lesssim N_{0}^{[\beta]+1}|y-x|^{|\beta|} \text{ for all populated }\beta,\label{ap1}\\
		|(\Gamma_{yx}^*-\textnormal{id})_\beta^\gamma|&\lesssim N_{0}^{[\beta]-[\gamma]}|y-x|^{|\beta|-|\gamma|}\text{ for all populated }(\gamma,\beta)\text{ with }|\gamma|<|\beta|.\label{ap3}
	\end{align}
Here $\lesssim$ means inequality up to a constant which only depends on $\alpha,\beta,\gamma$.
\end{theorem}

The reason why we have to exclude $\alpha\in \Q$ resides in the Schauder estimate in the
integration step, see Lemma \ref{lem:int}: 
Having $\alpha\notin \Q$ guarantees that $|\beta|\notin \N$ 
for all not purely-polynomial multi-indices $\beta$, cf. \eqref{hom}.


\subsection{Strategy of the proof}\label{subsec:strategy}
\mbox{}

We shall establish Theorem \ref{propanalyticbounds} by induction.
One might be tempted to use induction in $|\cdot|$. However, $|\beta|$ is oblivious to
the value of $\beta(0)$, cf. \eqref{hom}. Hence because of the $(k=0)$-term
in the definition \eqref{fm4} of $\Pi_{x\beta}^-$, the latter does 
not only depend on $\Pi_{x\gamma}$ with $|\gamma|<|\beta|$.
Hence we need to incorporate $\beta(0)$ into our ordering: Fixing a $\lambda\in (0,1)$ we define
\begin{equation}\label{LOord01}
	|\beta|_\prec := |\beta| + \lambda \beta(0)
\end{equation}
and will write $\beta'\prec\beta$ for $|\beta'|_\prec < |\beta|_\prec$. 
As for $|\cdot|$, the range for $|\cdot|_{\prec}$ is bounded from below and locally finite, which makes it suitable for an induction argument.
This augmented ordinal provides the strict inequalities which ensure 
the sufficiency of the induction hypothesis:
\begin{lemma}\label{lemnorm} For populated multi-indices
\begin{align}
&\beta = e_l + \beta_1 +...+\beta_{l+1} \implies \beta_1,...,\beta_{l+1}\prec \beta; 
\label{norm02}\\
&(\Gamma^*-\textnormal{id})_\beta^\gamma \neq 0 \implies \gamma\prec\beta.\label{norm03}
\end{align}
\end{lemma}
\begin{proof}
We start with \eqref{norm02}, distinguishing the cases $l\ge 1$ and $l=0$. For
$l\ge 1$, we appeal to \eqref{jan04} which implies $|\beta|$ $=|\beta_1|+...+|\beta_{l+1}|$. 
By \eqref{hom1} and $l+1\ge 2$ this yields $|\beta|>|\beta_1|,...,|\beta_{l+1}|$. 
In addition, by $\beta(0)=\beta_1(0) + ... + \beta_{l+1}(0)$ we also have 
$\beta(0)\ge \beta_1(0),...,\beta_{l+1}(0)$.
For $l=0$ we start from $\beta=e_0+\beta_1$ and thus have $|\beta|=|\beta_1|$ and
$\beta(0)>\beta_1(0)$. Here we need $\lambda>0$.
	
	\medskip
	
We now turn to the proof of \eqref{norm03} and consider the component-wise representation 
\eqref{sg7}. It follows from the definitions \eqref{der2} of $D^{(\0)}$ and \eqref{der3} of 
$D^{(\n)}$ for $\n\neq \0$ by induction in $l$ that
\begin{equation}\label{norm06}
(D^{(\n_1)}\cdots D^{(\n_l)})_{\beta_{l+1}}^\gamma \neq 0 
\implies \beta_{l+1}(0) + \#\{k|\n_{k}={\bf 0}\}\ge\gamma(0).
\end{equation}
%
Combining \eqref{norm06} with \eqref{sg16} we obtain
\begin{align*}
|\beta|_\prec = |\beta| + \lambda \beta(0) 
&\geq |\gamma|+\sum_{k=1}^l(|\beta_{k}|-|\n_{k}|) 
+\lambda\Big(\gamma(0)+\sum_{k=1}^l \beta_{k}(0)-\#\{k|\n_{k}={\bf 0}\}\Big)\\
&= |\gamma|_{\prec}+\sum_{k=1}^l (|\beta_{k}|_\prec - |\n_{k}| - \lambda \delta_\0^{\n_{k}}).
\end{align*}
It remains to argue that $|\beta_{k}|_\prec - |\n_{k}| - \lambda \delta_\0^{\n_{k}}>0$. 
In case of $\n_{k}=\0$, this is equivalent to $|\beta_{k}|_\prec - \lambda >0$;
this holds since $|\cdot|_{\prec}$ $\ge|\cdot|$ $\ge 1$, here we need $\lambda<1$.
In case of $\n_{k}\neq \0$, 
this is equivalent to $|\beta_{k}|_\prec-|\n_{k}|>0$, which is true by 
$|\cdot|_\prec\ge|\cdot|$ and the population condition \eqref{sg9}. 
\end{proof}

\medskip

We have to complement the induction by two further estimates. 
The first one is an estimate of the r.~h.~s., namely
\begin{align}\label{ap1min}
|\Pi_{x\beta t}^-(y)|\lesssim_{\beta,\alpha}N_{0}^{[\beta]+1}(\sqrt[4]{t})^{\alpha-2}
(\sqrt[4]{t}+|y-x|)^{|\beta| -\alpha}\nonumber\\ 
\text{ for all populated not purely polynomial }\beta;
\end{align}
the second one is an estimate of the characters $\pi_{xy}^{(\n)}$, namely
\begin{equation}\label{ap2}
|\pi_{yx\beta}^{(\n)}|\lesssim_{\beta,\alpha}N_{0}^{[\beta]+1}|y-x|^{|\beta|-|\n|}
\text{ for all populated }\beta\text{ and ${\bf n}$ with }|\n|<|\beta|.
\end{equation}
Let us now explain how the four estimates \eqref{ap1}, \eqref{ap3}, \eqref{ap1min},
and \eqref{ap2} are logically intertwined in the induction
over all populated multi-indices $\beta$ in terms of $|\beta|_{\prec}$; it is crucial
to distinguish the purely polynomial $\gamma$'s in \eqref{ap3} and treat them at the end
of a step.
The base case, which because of $\alpha>1$ just contains
$\beta=e_{(1,0)}$, is automatically included.

\begin{lemma}\label{lemalg}(Algebraic argument)
Suppose $\eqref{ap2}_{\prec \beta}$ holds. 
Then $\eqref{ap3}_\beta^\gamma$ holds for all not purely polynomial $\gamma$.
\end{lemma}

\begin{lemma}\label{lemrhs}(Reconstruction)
Let $\beta$ be not purely polynomial. 
Suppose $\eqref{ap1}_{\prec \beta}$ and $\eqref{ap3}_{\prec \beta}$ hold, 
and that $\eqref{ap3}_\beta^\gamma$ holds for all not purely polynomial $\gamma$. 
Then $\eqref{ap1min}_\beta$ holds.
\end{lemma}

\begin{lemma}\label{lemlhs}(Integration)
Let $\beta$ be not purely polynomial.
Suppose $\eqref{ap1min}_\beta$ holds. Then $\eqref{ap1}_\beta$ holds.
\end{lemma}
Note that $\eqref{ap1}_\beta$ for $\beta$ purely polynomial follows from the third item in \eqref{fm4}, since then $[\beta]+1=0$, so a proof is not required.
\begin{lemma}\label{lemtpa}(Three-point argument)
Suppose $\eqref{ap1}_{\preceq \beta}$ holds, 
and that $\eqref{ap3}_\beta^\gamma$ holds for all not purely polynomial $\gamma$. 
Then $\eqref{ap2}_\beta$ holds.
\end{lemma}

\begin{lemma}\label{lemalgII}
Suppose $\eqref{ap2}_{\beta}$ holds. 
Then $\eqref{ap3}_\beta^\gamma$ holds for all purely polynomial $\gamma$.
\end{lemma}

The last lemma is immediate from formula \eqref{sg2}.


\subsection{Algebraic argument (proof of Lemma \ref{lemalg})}
\mbox{}

\begin{proof}[Proof of Lemma \ref{lemalg}]
The argument relies on the exponential formula \eqref{sg3}. Fixing a $\gamma$ not purely polynomial,
we first show that the multi-indices in the component-wise version \eqref{sg7} of \eqref{sg3} 
only contribute when
\begin{align}\label{ud01}
|\beta_1|_\prec,...,|\beta_l|_\prec< |\beta|_\prec.
\end{align}
Indeed, by the additivity of $[\cdot]$ and the first item in \eqref{sg8} we have
\begin{align}\label{ud02}
([\beta_1]+1)+...+([\beta_l]+1)+[\gamma]=[\beta].
\end{align}
Since the populated $\gamma$ is not purely polynomial, we have $[\gamma]\geq 0$.
Because $[\cdot]+1$ is non-negative on populated multi-indices, 
\eqref{ud02} implies $[\beta_1]+1,\cdots,[\beta_l]+1\le[\beta]$.
%
Furthermore, by additivity of $\pol{\cdot}$ and the second item in \eqref{sg8} we have
\begin{align}\label{ud03}
(\pol{\beta_1}+|{\bf n}_1|)+\cdots+(\pol{\beta_l}+|{\bf n}_l|)+\pol{\gamma}=\pol{\beta}.
\end{align}
Because of $\pol{\cdot}\geq 0$, this implies
$\pol{\beta_1},...,\pol{\beta_l}\leq \pol{\beta}$.
Combined with the obvious $\beta_1(0),...,$ $\beta_l(0)$ $\leq$ $\beta(0)$, 
by \eqref{hom7} and definition \eqref{LOord01},
the two inequalities yield \eqref{ud01}.

\medskip

We now turn to the estimate \eqref{ap3}, based on \eqref{sg7}. We recall that the
sum over $l$, ${\bf n}_1,\cdots,{\bf n}_l$, and $\beta_1,\cdots,\beta_{l+1}$ is effectively
finite so that suffices to estimate the products $\pi_{\beta_1}^{({\bf n}_1)},\cdots,
\pi_{\beta_l}^{({\bf n}_l)}$. According to \eqref{ud01}, 
the induction hypothesis $\eqref{ap2}_{\prec\beta}$ is sufficient, and yields an estimate by
\begin{equation*}
N_0^{([\beta_1]+1)+\cdots+([\beta_l]+1)}
|y-x|^{(|\beta_1|-|\n_1|) +... + (|\beta_l|-|\n_l|)}. 
\end{equation*} 
By \eqref{ud02}, the exponent of $N_0$ assumes the desired value $[\beta]-[\gamma]$.
Likewise by \eqref{ud03} and \eqref{hom7}, the exponent of $|y-x|$ assumes
the desired value $|\beta|-|\gamma|$.
\end{proof}

\subsection{Reconstruction (proof of Lemma \ref{lemrhs})}
\mbox{}

Lemma \ref{lemrhs} relies on the following general regular reconstruction result.
In reconstruction, one is given a family $\{F_z\}$ of space-time Schwartz distributions,
indexed by a space-time point $z$, and the task is to construct a single distribution $F$
that is close to $F_z$ near $z$, which is possible provided the ``germs'' $F_z$ satisfy
a suitable continuity condition in $z$. The latter is best expressed in terms of its
convolution $F_{zt}$ in the active variable, see \eqref{rc02}. In our application, it suffices to
consider distributions that are actually continuous space-time functions, so that the task amounts 
to show that a mollification in the active variable, namely $F_{zt}(z)$, is close to the
``diagonal'' $F_z(z)$, or rather the mollification  $\int dy\psi_t(z-y)F_y(y)$ of the diagonal,
see \eqref{rc03}. In our application, the family $\{F_{xz}\}$
is indexed by second space-time (base) point $x$, which influences the topology of continuity
in $z$, see the r.~h.~s.~of \eqref{rc02}, and of \eqref{rc03}.
\begin{lemma}\label{lemrecon}
Suppose that for some exponents $\theta_1$, $\theta_2$, $\theta_3$ 
with $\theta_1+\theta_2 >0$ and $\theta_2,\theta_3\geq 0$ 
	\begin{equation}\label{rc02}
		|F_{xzt}(z) - F_{xyt}(z)| 
\leq (\sqrt[4]{t})^{\theta_1} (\sqrt[4]{t} + |y-z|)^{\theta_2} 
(\sqrt[4]{t} + |y-z| + |z-x|)^{\theta_3}.
	\end{equation}
Then
	\begin{equation}\label{rc03}
		|F_{xzt}(z) - \int dy \psi_t(z-y) F_{xy}(y)|\lesssim (\sqrt[4]{t})^{\theta_1 + \theta_2} (\sqrt[4]{t} + |z-x|)^{\theta_3}.
	\end{equation}
\end{lemma}

\begin{proof}[Proof of Lemma \ref{lemrhs}]
We start by getting rid of a couple of simple cases: For $\beta=0$
we have $\Pi^{-}_{x\beta}=\xi$ so that \eqref{ap1min} coincides with our assumption
\eqref{normxi} since $[\beta]=0$ and $|\beta|=\alpha$. For $\beta\in\mathbb{N}e_0$, we have 
$\Pi^{-}_{x\beta}=\partial_1^2\Pi_{x\beta-e_0}$. Since $\beta'=\beta-e_0\prec\beta$,
the induction hypothesis allows us to use \eqref{nrc05} below, which turns into
\eqref{ap1min} since $[\beta']=[\beta]$ and $|\beta'|=|\beta|$. 

\medskip

For the remaining indices $\beta$, we have $\beta(k)\not=0$ for some $k\ge 1$
and thus $|\beta|\ge 2\alpha$, which we use in form of
\begin{align}\label{ud05}
|\beta|\ge\alpha+1.
\end{align}
We now consider 
\begin{align}\label{ud14}
F_{xz}:= (\sum_k \z_k \Pi_x^k (z) \partial_1^2 \Pi_x)_\beta,
\end{align}
which is made such that 
%
\begin{equation}\label{nrc01}
F_{xy}(y)=\Pi_{x\beta}^{-}(y),
\end{equation}
see \eqref{fm4}.
Based on \eqref{ud14}, we shall establish
\begin{align}
	|F_{xzt}(z)| &\lesssim N_0^{[\beta]+1} (\sqrt[4]{t})^{\alpha-2} 
(\sqrt[4]{t} + |z-x|)^{|\beta|-\alpha},\label{nrc02}\\
        |F_{xzt}(z) - F_{xyt}(z)|&\lesssim N_0^{[\beta]+1} (\sqrt[4]{t})^{\alpha-2} 
(\sqrt[4]{t} + |y-z|) (\sqrt[4]{t} + |y-z| + |z-x|)^{|\beta|-\alpha-1}.\label{nrc04}
\end{align}
Since $\alpha>1$ and by \eqref{ud05}, $N_0^{-([\beta]+1)}F_{xz}$ meets the assumption \eqref{rc02}
of Lemma \ref{lemrecon}. In view of \eqref{nrc01}, the output \eqref{rc03} assumes the form of
\begin{align}
	|F_{xzt}(z) - \Pi_{x\beta t}^{-}(z)|
&\lesssim N_0^{[\beta]+1} (\sqrt[4]{t})^{\alpha -1} 
(\sqrt[4]{t} + |z-x|)^{|\beta|-\alpha-1}.\label{nrc03}
\end{align}
Combining \eqref{nrc02} with \eqref{nrc03} yields the desired \eqref{ap1min}. 


\medskip

\underline{Step 1.} Proof of 
\begin{equation}\label{nrc05}
\eqref{ap1}_{\preceq \beta'}\mbox{and}\
\eqref{ap3}_{\beta'}\;\implies\;|\partial_1^2 \Pi_{x\beta' t}(z)| 
\lesssim N_0^{[\beta']+1}(\sqrt[4]{t})^{\alpha-2} (\sqrt[4]{t} + |z-x|)^{|\beta'|-\alpha}. 
\end{equation}
Applying $\partial_1^2$ to \eqref{sg6} we obtain
\begin{equation}\label{ud07}
	\partial_1^2 \Pi_{x\beta' t}(z) = \sum_{\gamma'} (\Gamma_{xz}^*)_{\beta'}^{\gamma'}
\partial_1^2\Pi_{z\gamma'}(z)=
\sum_{\gamma'} (\Gamma_{xz}^*)_{\beta'}^{\gamma'}\int dy\partial_1^2\psi_t(z-y) 
\Pi_{z\gamma'}(y).
\end{equation}
Since the sum over $\gamma'$ is effectively finite, it suffices to estimate
every summand. On the first factor, we use $\eqref{ap3}_{\preceq \beta'}$;
on the second factor, because of \eqref{sg5}, we may use $\eqref{ap1}_{\preceq \beta'}$.
Appealing also to \eqref{ker02}, we see that the product is estimated by
\begin{align*}
N_0^{[\beta']-[\gamma']}|z-x|^{|\beta'|-|\gamma'|} 
\;N_0^{[\gamma']+1}(\sqrt[4]{t})^{|\gamma'|-2}.
\end{align*}
Provided $|\gamma'|\ge\alpha$, this expression is dominated by the r.~h.~s.~of \eqref{nrc05}.
Indeed, the sum in \eqref{ud07} is effectively restricted to $|\gamma'|\ge\alpha$,
since $\gamma'=e_{(1,0)}$ does not contribute in view of the third item in \eqref{fm4}, 
and since we have\footnote{ The restriction $\alpha<2$ is very convenient here, 
since otherwise the estimate would produce a contribution of the form $|z-x|^{|\beta|-2}$ 
which would not fit our general reconstruction argument. However, if $\alpha>2$ $\Pi^-$ 
is a function, so we may estimate it pointwise in an easier way.} $\alpha\le 2$. 
%

\medskip

\underline{Step 2.} Proof of \eqref{nrc02}. We write $F_{xzt}(z)$ component-wise:
\begin{equation*}
F_{xzt}(z) = \sum_{k\geq 0} \sum_{e_k + \beta_1 +...+\beta_{k+1}=\beta}
\Pi_{x\beta_1}(z)\cdots \Pi_{x\beta_k}(z) \partial_1^2 \Pi_{x\beta_{k+1} t}(z).
\end{equation*}
Since the sums are obviously effectively finite, it is enough to estimate every summand.
According to \eqref{norm02}, we may appeal to  $\eqref{ap1}_{\prec\beta}$
and $\eqref{nrc05}_{\prec\beta}$ to obtain an estimate by
\begin{equation*}
N_0^{[\beta_1]+1}|z-x|^{|\beta_1|}\;\cdots\;N_0^{[\beta_{k}]+1}|z-x|^{|\beta_k|}
\;N_0^{[\beta_{k+1}]+1}(\sqrt[4]{t})^{\alpha-2}(\sqrt[4]{t} + |z-x|)^{|\beta_{k+1}|-\alpha}.
\end{equation*}
Since by $e_k+\beta_1+\cdots+\beta_{k+1}=\beta$,
we have $([\beta_1]+1)+\cdots+([\beta_{k+1}]+1)=[\beta]+1$
and $|\beta_1|+...+|\beta_{k+1}| = |\beta|$, this term coincides with the r.~h.~s.~of 
\eqref{nrc02}.

\medskip

\underline{Step 3.} Proof of \eqref{nrc04}. By definition \eqref{ud14},
\begin{align*}
F_{xz}-F_{xy}&=\Big(\sum_{k\geq 1} \z_k (\Pi_x(z)-\Pi_x(y))\sum_{k'+ k'' = k-1}
\Pi_x^{k'}(z)\Pi_x^{k''}(y) \partial_1^2 \Pi_x\Big)_\beta \\
&\hspace*{-0.65cm}\underset{\eqref{sg6},\eqref{sg15}}{=}-\Big(\sum_{k\geq 1} 
\z_k \Gamma_{xz}^* \Pi_z(y) \sum_{k'+ k'' = k-1}\Pi_x^{k'}(z)\Pi_x^{k''}(y)\partial_1^2 \Pi_x
\Big)_\beta,
\end{align*}
which component-wise assumes the form
\begin{equation*}
	(F_{xz}-F_{xy})_t(z) = -\sum_{k\geq 0}\sum_{e_k + \beta_1 +... +\beta_{k+1} = \beta} \sum_\gamma (\Gamma_{xz}^*)_{\beta_1}^\gamma \Pi_{z\gamma}(y) \Pi_{x\beta_2}(*)\cdots \Pi_{x\beta_k}(*)\partial_1^2\Pi_{x\beta_{k+1} t}(z),
\end{equation*}
where the symbol $*$ stands for either $z$ or $y$. Again, it suffices to estimate each summand;
according to \eqref{norm02} and \eqref{norm03}, we may appeal to $\eqref{ap1}_{\prec\beta}$, 
$\eqref{ap3}_{\prec\beta}$ as well as $\eqref{nrc05}_{\prec\beta}$ to obtain an estimate by
\begin{align*}
N_0^{[\beta_1]-[\gamma]}|z-x|^{|\beta_1|-|\gamma|}\;
N_0^{[\gamma] +1}|y-z|^{|\gamma |}\; 
N_0^{[\beta_2]+1}|*-x|^{|\beta_2|}\;\cdots\;
N_0^{[\beta_k]+1}|*-x|^{|\beta_k|}\\
\times N_0^{[\beta_{k+1}]+1}(\sqrt[4]{t})^{\alpha-2}(\sqrt[4]{t}+|z-x|)^{|\beta_{k+1}|-\alpha}.
\end{align*}
Using $|*-x|\le\sqrt[4]{t}+|y-z|+|z-x|$, this is expression is
\begin{align*}
\le N_0^{([\beta_1]+1)+\cdots+([\beta_{k+1}]+1)}
|y-z|^{|\gamma|}
(\sqrt[4]{t})^{\alpha-2}(\sqrt[4]{t}+|y-z|+|z-x|)^{|\beta_1|+\cdots+|\beta_{k+1}|-|\gamma|-\alpha}.
\end{align*}
Using the same identities on $[\cdot]$ and $|\cdot|$ as in Step 2, this expression is
\begin{align*}
= N_0^{[\beta]+1}
|y-z|^{|\gamma|}
(\sqrt[4]{t})^{\alpha-2}(\sqrt[4]{t}+|y-z|+|z-x|)^{|\beta|-|\gamma|-\alpha},
\end{align*}
which is dominated by the r.~h.~s.~of \eqref{nrc04} because of $|\cdot|\ge 1$.
\end{proof}	

\medskip

\begin{proof}[Proof of Lemma \ref{lemrecon}]
All reconstruction arguments rely on a dyadic decomposition of the 
l.~h.~s.~of \eqref{rc03}. Following the arguments in \cite{OW19}, 
see e.~g. \cite[Subsection 5.5]{OW19}, we implement this decomposition with help
of the convolution semi-group \eqref{ud04}. Indeed, with help
of the diagonal evaluation operator $(EF_{x})(y):=F_{xy}(y)$ and its commutator
with convolution (always in the active variable), we write for any $\tau<t$
\begin{equation*}
F_{xzt}(z)-\int dy \psi_{t-\tau}(y-z) F_{xy \tau}(y)=([(\cdot)_{t-\tau},E](\cdot)_\tau F_x)(z).
\end{equation*}
Specifying to a dyadic fraction $\tau$ of $t$ (i.~e. $\tau \in 2^{-\N}t$)
and appealing once more to \eqref{ud04}, we obtain from telescoping
%
\begin{equation*}
[(\cdot)_{t-\tau},E](\cdot)_\tau
=\sum_{\substack{\tau\leq s < t\\ s\textnormal{ dyadic fraction of $t$}}} 
(\cdot)_{t-2s}[(\cdot)_s,E](\cdot)_s.
\end{equation*}
By our assumption of qualitative continuity, this combines to
\begin{equation}\label{tel01}
F_{xzt}(z)-\int dy \psi_{t}(y-z) F_{xy}(y)
=\sum_{\substack{s < t\\ s\textnormal{ dyadic fraction of $t$}}}([(\cdot)_s,E]F_{xs})_{t-2s}(z).
\end{equation}

\medskip

Equipped with the decomposition \eqref{tel01}, we now estimate
its constituents. Writing
%
\begin{equation*}
([(\cdot)_s,E]F_{xs})(z) = \int dy \psi_s (z-y) (F_{xy s}(y) - F_{xz s}(y)),
\end{equation*}
we obtain by assumption \eqref{rc02} and the moment bounds \eqref{ker02} (also
using $\theta_2,\theta_3\ge 0$),
\begin{align}\label{ud06}
|([(\cdot)_s,E]F_{xs})(z)|&
\leq \int dy |\psi_s (z-y)| (\sqrt[4]{s})^{\theta_1} 
(\sqrt[4]{s} + |z-y|)^{\theta_2} (\sqrt[4]{s} + |z-y| + |z-x|)^{\theta_3}\nonumber\\
&\lesssim (\sqrt[4]{s})^{\theta_1 + \theta_2} (\sqrt[4]{s} + |z-x|)^{\theta_3}.
\end{align}
Writing 
\begin{align*}
([(\cdot)_s,E]F_{xs})_{t-2s}(z)=\int dy\psi_{t-2s}(z-y)
([(\cdot)_s,E]F_{xs})(y),
\end{align*}
we obtain from \eqref{ud06} (with $z$ replaced by $y$) and once more \eqref{ker02}
\begin{align*}
|([(\cdot)_s,E]F_{xs})_{t-2s}(z)|\lesssim (\sqrt[4]{s})^{\theta_1 + \theta_2} 
(\sqrt[4]{t} + |z-x|)^{\theta_3}.
\end{align*}
Since by assumption $\theta_1+\theta_2>0$, when summing over dyadic fractions $s<t$,
the largest scale matters, and thus obtain \eqref{rc03} from \eqref{tel01}.
%
%
%
%
\end{proof}


\subsection{Integration step (proof of Lemma \ref{lemlhs})}
\mbox{}

The integration step is an instance of Schauder theory on some regularity level $\eta$
for the operator $\partial_2-\partial_1^2$ of (parabolic) order 2. 
It states that $(\partial_2-\partial_1^2)^{-1}$ improves regularity by 2 units, from
$\eta-2$ to $\eta>0$.
As for all Schauder theory, $\eta$ is not allowed to be an integer, see \eqref{ud13}.
Lemma \ref{lem:int} differs from a standard Schauder result in that 
it propagates a H\"older regularity that is localized at some base point $x$, 
see the output \eqref{output}; the input \eqref{gi06} being
a localized version of the characterization of (negative) H\"older spaces via convolution, 
cf.~\eqref{normxi}.
Like the classical proof, ours is based on
a kernel representation of $(\partial_2-\partial_1^2)^{-1}$, 
and relies on distinguishing a near-field and a far-field part. 
However, our kernel representation is not based on the
fundamental solution of $\partial_2-\partial_1^2$, but on our semi-group generated by
the square modulus of that operator, see \eqref{gi07}.
\begin{lemma}\label{lem:int}
Let $\eta>0$ satisfy
\begin{align}\label{ud13}
\eta\not\in\mathbb{N}.
\end{align}
Suppose the smooth function $g$ satisfies for some exponent $0<\theta\le\eta$ and some point $x$
\begin{equation}\label{gi06}
|g_t (y)| \leq (\sqrt[4]{t})^{\theta-2}(\sqrt[4]{t}+|y-x|)^{\eta-\theta}.
\end{equation}
Then 
\begin{equation}\label{gi07}
u=-\int_0^\infty dt (1-\mathrm{T}_x^k)(\partial_2 + \partial_1^2)g_t,
\end{equation}
where $\mathrm{T}_x^k$ stands for taking the Taylor polynomial of (parabolic) order 
$k = \lfloor\eta\rfloor$ centered at $x$, is well-defined and satisfies the estimate
\begin{equation}\label{output}
|u(y)|\lesssim |y-x|^{\eta}.
\end{equation}
Moreover, $u$ solves \eqref{jan08}.
\end{lemma}

\begin{proof}[Proof of Lemma \ref{lemlhs}]
For $\beta$ not purely polynomial and thus
$[\beta]\ge 0$, we may use Lemma \ref{lem:int} with $\eta=|\beta|\ge\alpha$, 
which satisfies \eqref{ud13} by our assumption $\alpha\not\in\mathbb{Q}$, see \eqref{hom7}, and  $g=N_0^{-([\beta]+1)}\Pi_{x\beta}^-$. The input \eqref{gi06} is 
provided by $\eqref{ap1min}_\beta$ with $\theta=\alpha\in(0,\eta]$. 
By the uniqueness statement of Lemma \ref{lem:uni}, $u=N_0^{-([\beta]+1)}\Pi_{x\beta}$ so that
the output \eqref{output} amounts to the desired $\eqref{ap1}_\beta$.
\end{proof}

\begin{proof}[Proof of Lemma \ref{lem:int}]
In preparation of the estimates of the near and far-field contributions to \eqref{gi07}
we note that by \eqref{ud04} we may write
$\partial^{\bf n}g_t(y)$ $=\int dz\partial^{\bf n}\psi_\frac{t}{2}(y-z)g_\frac{t}{2}(z)$, so that
by \eqref{ker02}, our assumption \eqref{gi06} can be upgraded to
\begin{align}\label{ud10}
|\partial^{\bf n}g_t(y)|\lesssim(\sqrt[4]{t})^{\theta-2-|{\bf n}|}
(\sqrt[4]{t}+|y-x|)^{\eta-\theta}.
\end{align}
				
\medskip
		
\underline{Step 1.} Far-field estimate:
\begin{equation}\label{gi04}
\int_{|y-x|^4}^{\infty} dt \big|(1-\mathrm{T}_x^k)(\partial_2 + \partial_1^2)g_t(y)\big|
\lesssim |y-x|^{\eta}.
\end{equation}
From the integral representation of Taylor's remainder 
for a function $h$ of a single variable $s$
\begin{equation*}
h(1) = \sum_{j=0}^k \frac{1}{j!} \frac{d^j h}{ds^j}(0) 
+ \int_0^1 ds \frac{1}{k!} (1-s)^k \frac{d^{k+1}h}{ds^{k+1}}(s),
\end{equation*}
which we apply to $h(s) := (1-\mathrm{T}_x^k)(\partial_2 + \partial_1^2)g_t (sy + (1-s)x)$, 
we obtain by $|{\bf n}|\ge n_1+n_2$
\begin{equation*}
\begin{split}
&(1-\mathrm{T}_x^k)(\partial_2 + \partial_1^2)g_t(y)
=\sum_{\substack{|\mathbf{n}| > k \\ n_1+n_2 \leq k}}(y-x)^{\mathbf{n}} 
\frac{1}{\mathbf{n}!}\partial^{\mathbf{n}}(\partial_2 + \partial_1^2)g_t(x)\\
&+ \sum_{n_1 + n_2=k+1} (y-x)^{\mathbf{n}}\int_0^1 ds (1-s)^k \frac{k+1}{\mathbf{n}!} 
\partial^{\mathbf{n}}(\partial_2 + \partial_1^2) g_t(sy+(1-s)x).
\end{split}
\end{equation*}
Estimate \eqref{ud10} thus yields
\begin{align*}
|(1-\mathrm{T}_x^k)(\partial_2 + \partial_1^2)g_t(y)|
&\lesssim \sum_{\substack{|\mathbf{n}|>k\\n_1+n_2 \leq k+1}} |y-x|^{|\mathbf{n}|}
\sup_{|z-x|\leq |y-x|} |\partial^{\mathbf{n}}(\partial_2 + \partial_1^2) g_t(z)|\\
&\lesssim \sum_{\substack{|\mathbf{n}| > k \\ n_1+n_2 \leq k+1}} |y-x|^{|\mathbf{n}|}
(\sqrt[4]{t})^{\theta-4-|\mathbf{n}|}(\sqrt[4]{t}+|y-x|)^{\eta-\theta}.
\end{align*}
Since by definition $k=\lfloor\eta\rfloor$, we have $\eta-|\mathbf{n}|<0$, 
every summand is integrable for $t\uparrow\infty$. 
Integrating $t$ from $|y-x|^4$ onwards yields \eqref{gi04}.
		
\medskip
		
\underline{Step 2.} Near-field estimate:
\begin{equation}\label{gi05}
\int_{0}^{|y-x|^4} dt \big|(1-\mathrm{T}_x^k)(\partial_2 + \partial_1^2)g_t(y)\big|
\lesssim |y-x|^{\theta + \eta}.
\end{equation}
We use the triangle inequality in the form of
\begin{align*}
|(1-\mathrm{T}_x^k)(\partial_2 + \partial_1^2)g_t(y)| 
&\le|(\partial_2 + \partial_1^2)g_t(y)|+\sum_{|\mathbf{n}|\leq k}
\frac{1}{|{\bf n}!|}|y-x|^{|\mathbf{n}|}|\partial^{\mathbf{n}}(\partial_2 + \partial_1^2)g_t(x)|\\
&\hspace*{-0.25cm}\underset{\eqref{ud10}}{\lesssim}
(\sqrt[4]{t})^{\theta-4}(\sqrt[4]{t}+|y-x|)^{\eta-\theta}
+\sum_{|\mathbf{n}|\leq k}
|y-x|^{|\mathbf{n}|}(\sqrt[4]{t})^{\eta-4-|{\bf n}|}.
\end{align*}
%
Both expressions are integrable for $t\downarrow 0$, 
the former due to $\theta>0$, 
the latter due to $\eta>k$, which follows from the definition $k=\lfloor\eta\rfloor$ and
\eqref{ud13}. Integrating $t$ up to $|y-x|^4$ yields \eqref{gi05}.

\medskip

\underline{Step 3.} We claim that $u$ solves \eqref{jan08}. By the decay for $t\uparrow\infty$
established in Step 1 and by the (qualitative) smoothness of $g$,
which takes care of $t\downarrow0$, we have that also $u$ is smooth and
\begin{align}\label{ud11}
\partial^{\bf n}u
=-\int_0^\infty dt\partial^{\bf n}(1-{\rm T}_x^k)(\partial_2+\partial_1^2)g_t
=-\int_0^\infty dt(1-{\rm T}_x^{k-|{\bf n}|})\partial^{\bf n}(\partial_2+\partial_1^2)g_t,
\end{align}
with the integral converging absolutely and pointwise.
In particular, the presence of the Taylor term in representation \eqref{ud11} implies
$\partial^{\bf n}u(x)=0$ for $|{\bf n}|\le k$, as desired.
Since $(\cdot)_t$ is the semigroup generated by
$-(\partial_2-\partial_1^2)(\partial_2+\partial_1^2)$, \eqref{ud11} implies
$\partial^{\bf n}(\partial_2-\partial_1^2)u$ $=-\int_0^\infty dt$
$(1-{\rm T}_x^{k-|{\bf n}|-2})\partial^{\bf n}\partial_tg_t$ and thus
\begin{align}\label{ud12}
\partial^{\bf n}(\partial_2-\partial_1^2)u
=-\int_0^\infty dt\partial_t\partial^{\bf n}g_t\quad\mbox{for}\;|{\bf n}|>k-2.
\end{align}
By \eqref{ud10},  $\partial^{\bf n}g_t$ is (qualitatively) asymptotically dominated
by $(\sqrt[4]{t})^{\eta-|{\bf n}|-2}$, which by definition $k=\lfloor\eta\rfloor$ 
amounts to a decay. Hence \eqref{ud12} reduces to
$\partial^{\bf n}(\partial_2-\partial_1^2)u$ $=\partial^{\bf n}g$ for all $|{\bf n}|>k-2$,
which yields that $p:=(\partial_2-\partial_1^2)u-g$ is a polynomial of degree $\le k-2$, as desired.
\end{proof}


\subsection{Three-point argument (proof of Lemma \ref{lemtpa})}
\mbox{}
\begin{proof}[Proof of Lemma \ref{lemtpa}]
From \eqref{sg6} and \eqref{sg15} we deduce the following identity involving three points $x,y,z$
\begin{equation*}
\Pi_{x}(z)-\Pi_{x}(y)-\Pi_{y}(z)=(\Gamma_{xy}^*-{\rm id})\Pi_{y}(z).
\end{equation*}
With help of \eqref{sg2} and the third item in \eqref{fm4}, we rewrite this as
\begin{equation*}
\Pi_{x\beta}(z)-\Pi_{x\beta}(y)-\Pi_{y\beta}(z)
=\sum_{\gamma\neq \text{p.p.}}(\Gamma_{xy}^*-{\rm id})_\beta^\gamma\Pi_{y\gamma}(z)
+\sum_{\n\not=\0}\pi_{xy\beta}^{(\n)}(z-y)^{\n}.
\end{equation*}
By $\eqref{ap1}_\beta$ and $\eqref{ap3}_\beta^\gamma$ for $\gamma$ not purely polynomial, 
this implies
\begin{align*}
\bigg|\sum_{\n\neq \0} \pi_{xy\beta}^{(\n)}(z-y)^\n\bigg| 
&\lesssim N_0^{[\beta]+1} (|x-z|^{|\beta|} + |y-x|^{|\beta|} + |z-y|^{|\beta|} 
+ \sum_{\substack{|\gamma|<|\beta|\\ \gamma\neq \text{p.p.}}}
|y-x|^{|\beta|-|\gamma|}|z-y|^{|\gamma|}\\
&\lesssim N_0^{[\beta]+1}(|y-x|^{|\beta|} + |z-y|^{|\beta|}).
\end{align*}
Writing $z-y=|y-x|\hat{z}$, the above estimate can be reformulated as
\begin{equation*}
\left|\sum_{\n\not=0}\frac{\pi_{xy\beta}^{(\n)}}{|y-x|^{|\beta|-|\n|}}\hat{z}^{\n}\right|
\lesssim N_0^{[\beta]+1}\quad\mbox{for all}\;\hat z\;\mbox{with}\;|\hat z|\le 1.
\end{equation*}
Since by \eqref{sg9}, the l.~h.~s.~is a polynomial in $\hat z$ of degree $<|\beta|$,
this (locally) uniform estimate implies\footnote{e.~g.~by appealing to equivalence of norms} 
an estimate on its coefficients 
\begin{equation*}
\frac{|\pi_{xy\beta}^{(\n)}|}{|y-x|^{|\beta|-|\n|}}\lesssim N_0^{[\beta]+1},
\end{equation*}
which coincides with $\eqref{ap2}_\beta$.
\end{proof}

\bibliographystyle{abbrv}
\bibliography{bib_thesis}{}

\begin{thebibliography}{1}

\bibitem{Abe80}
E.~Abe.
\newblock {\em Hopf algebras}, volume~74 of {\em Cambridge Tracts in
  Mathematics}.
\newblock Cambridge University Press, Cambridge-New York, 1980.
\newblock Translated from the Japanese by Hisae Kinoshita and Hiroko Tanaka.

\bibitem{Che22}
I.~Chevyrev.
\newblock Hopf and pre-{L}ie algebras in regularity structures.
\newblock Preprint \emph{arXiv:2206.14557}, 2022.

\bibitem{Ha14}
M.~Hairer.
\newblock A theory of regularity structures.
\newblock {\em Invent. Math.}, 198(2):269--504, 2014.

\bibitem{Ha16}
M.~Hairer.
\newblock Regularity structures and the dynamical {$\Phi^4_3$} model.
\newblock In {\em Current developments in mathematics 2014}, pages 1--49. Int.
  Press, Somerville, MA, 2016.

\bibitem{Kry96}
N.~V. Krylov.
\newblock {\em Lectures on elliptic and parabolic equations in {H}\"{o}lder
  spaces}, volume~12 of {\em Graduate Studies in Mathematics}.
\newblock American Mathematical Society, Providence, RI, 1996.

\bibitem{LOT21}
P.~Linares, F.~Otto, and M.~Tempelmayr.
\newblock The structure group for quasi-linear equations via universal
  enveloping algebras.
\newblock Preprint \emph{arXiv:2103.04187}, 2021.

\bibitem{LOTT21}
P.~Linares, F.~Otto, M.~Tempelmayr, and P.~Tsatsoulis.
\newblock A diagram-free approach to the stochastic estimates in regularity
  structures.
\newblock Preprint \emph{arXiv:2112.10739}, 2021.

\bibitem{OSSW21}
F.~Otto, J.~Sauer, S.~Smith, and H.~Weber.
\newblock A priori bounds for quasi-linear {SPDEs} in the full sub-critical
  regime.
\newblock Preprint \emph{arXiv:2103.11039}, 2021.

\bibitem{OW19}
F.~Otto and H.~Weber.
\newblock Quasilinear {SPDE}s via rough paths.
\newblock {\em Arch. Ration. Mech. Anal.}, 232(2):873--950, 2019.

\end{thebibliography}

\end{document}